\documentclass{amsart}%
\usepackage{palatino, mathpazo}
\usepackage{amsfonts}
\usepackage{amsmath}
\usepackage{amssymb,latexsym,xcolor}
\usepackage{graphicx}
\usepackage[mathscr]{eucal}
\usepackage{harpoon}
\usepackage{amssymb}
\usepackage[
linktocpage=true,colorlinks,citecolor=magenta,linkcolor=blue,urlcolor=magenta]{hyperref}
\newtheorem{theorem}{Theorem}[section]

\newtheorem{corollary}[theorem]{Corollary}
\newtheorem{definition}[theorem]{Definition}
\newtheorem{example}[theorem]{Example}
\newtheorem{lemma}[theorem]{Lemma}
\newtheorem{Proposition}[theorem]{Proposition}
\newtheorem{Theorem}{Theorem}

\theoremstyle{remark}
\newtheorem{remark}[theorem]{Remark}

\numberwithin{equation}{section}

\def\R{\mathbb R}

\def\bb{\beta}
\def\dd{\delta}

\def\e{\varepsilon}
\def\Om{\Omega}

\def\iy{\infty}

\def\la{\lambda}

\def\bt{\begin{theorem}}
\def\et{\end{theorem}}
\def\bl{\begin{lemma}}
\def\el{\end{lemma}}
\def\bd{\begin{definition}}
\def\ed{\end{definition}}
\def\bc{\begin{corollary}}
\def\ec{\end{corollary}}
\def\bprop{\begin{Proposition}}
\def\eprop{\end{Proposition}}
\def\bp{\begin{proof}}
\def\ep{\end{proof}}
\def\bx{\begin{example}}
\def\ex{\end{example}}
\def\br{\begin{remark}}
\def\er{\end{remark}}
\def\be{\begin{equation}}
\def\ee{\end{equation}}
\def\bal{\begin{align}}
\def\bn{\begin{enumerate}}
\def\en{\end{enumerate}}
\def\eal{\end{align}}

\def\bg{\begin{align*}}
\def\eg{\end{align*}}
\def\bcs{\begin{cases}}
\def\ecs{\end{cases}}

\def\RNum#1{\uppercase\expandafter{\romannumeral #1\relax}}

\def\RN{\mathbb R^N}

\def\bean{\begin{eqnarray*}}
\def\eean{\end{eqnarray*}}

\begin{document}
\title[A critical elliptic system in a punctured ball]{Qualitative analysis of positive singular solutions for a critical elliptic system in a punctured ball}
\author{Zhijie Chen}
\address{Department of Mathematical Sciences, Yau Mathematical Sciences Center,
Tsinghua University, Beijing, 100084, China}
\email{zjchen2016@tsinghua.edu.cn}

\author{Hui Yang}
\address{School of Mathematical Sciences, Shanghai Jiao Tong University, Shanghai 200240, China}
\email{hui-yang@sjtu.edu.cn}

\author{Junzhao Yu}
\address{Department of Mathematical Sciences,
Tsinghua University, Beijing, 100084, China}
\email{yjz21@mails.tsinghua.edu.cn}


\begin{abstract}
We study qualitative properties of positive singular solutions to a weakly coupled elliptic system with critical exponents in a punctured ball. 
We give a sharp criterion on the removablity of the isolated singularity.
We prove that semi-singular solutions (i.e., solutions with one component bounded near the singularity and the other component unbounded near the singularity) do not exist for the dimensions $N\geq 4$ but do exist for $N=3$. Asymptotic symmetry and sharp pointwise estimates are also proved for singular solutions. These generalizes some classical results of (Caffarelli, Gidas and Spruck, Comm. Pure Appl. Math, 1989) to the weakly coupled system. Moreover, for the weakly coupled system, we demonstrate a novel phenomenon that does not arise in the scalar equation.
\end{abstract}
\maketitle

\section{Introduction}

In this paper, we study positive singular solutions of the following weakly coupled elliptic system with critical exponents in a punctured ball:
\be\label{eq0}
\begin{cases}
\begin{split}-\Delta u =
\mu_1 u^{2^\ast-1}+\beta u^{\frac{2^\ast}{2}-1} v^{\frac{2^\ast}{2}}\\
-\Delta v =\mu_2 v^{2^\ast-1}+\beta v^{\frac{2^\ast}{2}-1}u^{\frac{2^\ast}{2}}\\
\end{split}  \quad   \text{in $B_2\setminus\{0\}$},\\
u, v> 0 \quad\hbox{and}\quad
u, v\in C^2(B_2\setminus\{0\}),\end{cases}\ee
where $\mu_1, \mu_2$, $\bb$ are all positive constants, $B_r:=\{x\in\RN\,|\,|x|<r\}$ is a ball, $N\ge 3$ and $2^\ast:=\frac{2N}{N-2}$ is the Sobolev critical exponent.

System (\ref{eq0}) is the critical case of the following general problem
\be\label{eq002}
\begin{cases}
\begin{split}-\Delta u =\la_1 u +
\mu_1 u^{2p-1}+\beta u^{p-1}v^{p}\\
-\Delta v =\la_2 v+\mu_2 v^{2p-1}+\beta v^{p-1} u^{p}\\
\end{split}  \quad   \text{in $\Om$},\\
u, v> 0 \,\,\hbox{in $\Om$},\end{cases}\ee
where $\Om\subset\RN$ is a domain, $p>1$ and $2p\le 2^\ast$ if $N\ge 3$.
In the case $p=2$, the cubic system (\ref{eq002}) arises from
nonlinear optics and Bose-Einstein condensates. We refer the reader to the survey articles \cite{F, KL},  which
also contain information about the physical relevance of non-cubic nonlinearities.
Since the pioneering work of Lin and Wei \cite{LW1}, the existence and properties of solutions of \eqref{eq002} with general $2<2p\leq 2^\ast$ has been widely studied in the past two decades, and it seems impossible for us to list all the references. See e.g. \cite{ChenLi, CZ, CZ2, GL, LW, long2020, MMP, NTTV, PS, PT, QS, S, TV, wei2020} and references therein.

The study of singular solutions has been a central topic in PDEs for a long time (cf. Aviles \cite{Aviles} and Lions \cite{Lions}). 
One of our motivations of studying \eqref{eq0} comes from the seminal work of Caffarelli, Gidas and Spruck \cite{CGS}, where among other things, they studied local properties of solutions for the scalar equation with critical exponent
\be\label{eq1-2}
\begin{cases}
-\Delta u = u^{2^\ast-1}  \quad   \text{in $B_2\setminus\{0\}$},\\
u> 0 \quad\hbox{and}\quad
u\in C^2(B_2\setminus\{0\}).\end{cases}\ee
Let $\mathbb{S}:=\{\theta\in \RN\;|\; |\theta|=1\}$. Define the spheical average
\be\label{eq1-3}\bar{u}(x):=\frac1{|\mathbb{S}|}\int_{\mathbb{S}}u(|x|\theta)\,d\theta,\ee
and write $\bar{u}(x)=\bar{u}(|x|)$ for convenience. By the Pohozaev identity, we have for $0<s<r$ that
$P(r; u)-P(s; u)=0$,
where
$$P(r; u):=\int_{\partial B_r}\left(\frac{N-2}{2}u\frac{\partial u}{\partial \nu}-\frac{r}{2}|\nabla u|^2+r\left|\frac{\partial u}{\partial \nu}\right|^2+\frac{r}{2^\ast}u^{2^\ast}\right)d\sigma,$$
and $\nu$ is the unit outer normal of $\partial B_r$. Hence $P(r; u)$ is a constant independent of $r$, and we denote this Pohozaev invariant by $P(u)$. Then Caffarelli, Gidas and Spruck \cite{CGS} proved the following remarkable results.

\begin{Theorem}\cite{CGS}\label{thm-A} Let $u$ be a solution of \eqref{eq1-2}. Then $P(u)\leq 0$ and 
$$u(x)=\bar u(|x|)(1+O(|x|))\quad\text{for $|x|>0$ small}.$$
Furthermore,
\begin{itemize}
\item[(1)] $P(u)=0$ if and only if the singularity $0$ is removable for $u$, i.e., $u\in C^2(B_2)$.
\item[(2)] If $P(u)<0$, then there exist $C_2>C_1>0$ such that
$${C}_1|x|^{\frac{2-N}{2}}\le u(x)\le {C}_2|x|^{\frac{2-N}{2}},\quad\text{for $|x|>0$ small},$$
and there is a singular solution $v(x)=v(|x|)$ of
\be\label{eq1-4}
\begin{cases}
-\Delta u = u^{2^\ast-1}  \quad   \text{in $\mathbb R^N\setminus\{0\}$},\\
u> 0,\; u(x)\to+\infty \;\hbox{as $x\to 0$, }\quad
u\in C^2(\mathbb R^N\setminus\{0\}),\end{cases}\ee
such that
$$u(x)=v(|x|)(1+o(1))\quad\text{for $|x|>0$ small}.$$
\end{itemize}
\end{Theorem}

All singular solutions of \eqref{eq1-4} were also completely classified in \cite{CGS} by ODE analysis.
Since \cite{CGS}, 
equation (\ref{eq1-2}) and similar problems have received great interest and have been widely studied, and it seems impossible for us to list all the references. See e.g. \cite{Ado, ADGW, ACDFGW, CJSX, COS, CL1, CL2, CL3, ChenLin1, FK, GKS, GHWW, HanQ, HXZ, KMPS, JX, Li1996, TZ, YZ} and references therein.
For example, Caffarelli et al \cite{CJSX} studied the fractional order equation $(-\Delta)^\sigma u=u^{\frac{N+2\sigma}{N-2\sigma}}$ with an isolated singularity, where $\sigma\in (0,1)$; 
Jin and Xiong \cite{JX} studied the higher order equation $(-\Delta)^m u=u^{\frac{N+2m}{N-2m}}$ with an isolated singularity;
Caju, do \'{O} and Santos \cite{COS}, and Ghergu, Kim and Shahgholian \cite{GKS} both studied the elliptic system \be\label{eq05}-\Delta {\mathbf u}=|\mathbf u|^{2^\ast-2}\mathbf u\ee with an isolated singularity, where $\mathbf u=(u_1,\cdots, u_m)$, and so on.

Concerning the weakly coupled system \eqref{eq0} with an isolated singularity, 
 Chen and Lin \cite{ChenLin1} studied the problem in $\mathbb{R}^N\setminus\{0\}$, namely
\be\label{eq1-5}
\begin{cases}
\begin{split}-\Delta u =
\mu_1 u^{2^\ast-1}+\beta u^{\frac{2^\ast}{2}-1} v^{\frac{2^\ast}{2}}\\
-\Delta v =\mu_2 v^{2^\ast-1}+\beta v^{\frac{2^\ast}{2}-1}u^{\frac{2^\ast}{2}}\\
\end{split}  \quad   \text{in $\RN\setminus\{0\}$},\\
u, v> 0 \quad\hbox{and}\quad
u, v\in C^2(\RN\setminus\{0\}),\end{cases}\ee
which can be seen as a generalization of the scalar equation \eqref{eq1-4}.
They proved that any solution of \eqref{eq1-5} is radially symmetric, so \eqref{eq1-5} could also be studied by ODE analysis. However, the corresponding ODE of \eqref{eq1-5} is the following system\begin{equation}\label{eq30-017}
\begin{cases}
v_1''-\dd^2v_1+\mu_1v_1^{2^\ast-1}+\bb v_1^{\frac{2^\ast}{2}-1}v_2^{\frac{2^\ast}{2}}
=0,\\
v_2''-\dd^2v_2+\mu_2v_2^{2^\ast-1}+\bb v_2^{\frac{2^\ast}{2}-1}v_1^{\frac{2^\ast}{2}}
=0,
\end{cases} t\in\mathbb{R}, 
\end{equation} which proposes many additional difficulties compared to the corresponding ODE 
\begin{equation}\label{eq-ode1}
v''-\dd^2v+v^{2^\ast-1}
=0,\quad t\in\mathbb{R},
\end{equation}
of the scalar equation \eqref{eq1-4}. For example, when the Pohozaev invariant of  \eqref{eq1-4} is negative, it is easy to see that the corresponding solution $v$ of \eqref{eq-ode1} is a periodic function, from which all singular solutions of \eqref{eq1-4} can be completely classified; see \cite{CGS}. However,
when the Pohozaev invariant $D(u,v)<0$ (see \eqref{eq7} below for the definition), it is too difficult to prove whether the corresponding solution $(v_1, v_2)$ of \eqref{eq30-017} is periodic or not. Indeed, some numerical simulations suggest that this should not be true, that is, the structure of singular solutions of \eqref{eq1-5} might be very complicated.
Due to this reason, although all entire solutions (i.e., $u, v\in C^2(\mathbb{R}^N)$) of \eqref{eq1-5}  have been completely classified in \cite{ChenLi, GL}, the classification of singular solutions of \eqref{eq1-5} was not solved in \cite{ChenLin1} and still remains open. 

Remark that system \eqref{eq0} (i.e., the same system in a punctured ball) was not studied in \cite{ChenLin1} and also remains open so far. Compared to \eqref{eq1-5}, the additional difficulty for \eqref{eq0} is that solutions are no longer necessarily radially symmetric, which makes the study of \eqref{eq0} more difficult. On the other hand, the system \eqref{eq05} studied in \cite{COS, GKS} are fully symmetric, so
$$\nabla (u_j\nabla u_i-u_i\nabla u_j)\equiv 0,\quad\forall \, i\neq j.$$
This property gives a new invariant besides the Pohozaev invariant and plays a crucial role in the study of \eqref{eq05}. Clearly such property does not hold for \eqref{eq0} unless $N=4$ with $\mu_1=\mu_2=\beta$, so the main ideas of  \cite{COS, GKS} can not work for \eqref{eq0}.

The purpose of this paper is to generalize some results of Theorem \ref{thm-A} to the system \eqref{eq0} and explore new phenomena by developing new ideas. In particular, we will see Theorems \ref{thm:N3-semi} and \ref{th120} below that novel phenomena occur for singular solutions of the system \eqref{eq0}.

Given a solution $(u,v)$ of \eqref{eq0}, there are three possibilities.
\begin{itemize}
\item Both $u$ and $v$ are bounded near $x=0$, namely the singularity $0$ is removable for both $u$ and $v$ and so $u, v\in C^2(B_2)$.
\item One of $u, v$ is bounded near $x=0$ and the other is unbounded near $x=0$, and we call this type of solutions to be semi-singular solutions.
\item Both $u$ and $v$ are unbounded near $x=0$, and we call this type of solutions to be both-singular solutions.
\end{itemize}

Multiplying the first equation of (\ref{eq0}) by $x\cdot\nabla u$, the second equation by $x\cdot\nabla v$, and integrating over $B_r\setminus B_s$, we easily obtain the following Pohozaev identity
\be\label{eq6}D(r; u, v)=D(s; u, v),\ee
where
{\allowdisplaybreaks
\begin{align}\label{eq7}
&D(r; u, v):=\int_{\partial B_r}\Bigg[\frac{N-2}{2}\left(u\frac{\partial u}{\partial \nu}+v\frac{\partial v}{\partial \nu}\right)-\frac{r}{2}(|\nabla u|^2+|\nabla v|^2)\nonumber\\
&\quad+r\left|\frac{\partial u}{\partial \nu}\right|^2+r\left|\frac{\partial v}{\partial \nu}\right|^2+\frac{r}{2^\ast}\left(\mu_1 u^{2^\ast}+\mu_2 v^{2^\ast}+2\bb u^{\frac{2^\ast}{2}}v^{\frac{2^\ast}{2}}\right)\Bigg]d\sigma.
\end{align}
}%
Hence $D(r; u, v)$ is a constant independent of $r$, and we denote this Pohozaev invariant by $D(u, v)$. Let $\bar u(|x|)$, $\bar v(|x|)$ be the spherical averages of $u$ and $v$ respectively, defined by \eqref{eq1-3}.
Our first result is as follows.

\begin{theorem}\label{th1}Let $N\geq 3$, $\mu_1, \mu_2, \bb>0$, and $(u, v)$ be a solution of \eqref{eq0}. Then
\begin{itemize}
\item[(1)] There is a constant $C_0>0$ such that
$$u(x), v(x)\leq C_0|x|^{-\frac{N-2}{2}},\quad\text{for $|x|>0$ small}.$$
\item[(2)]  
\be\label{eq00}u(x)=(1+O(|x|))\bar{u}(|x|),\; v(x)=(1+O(|x|))\bar{v}(|x|),\;\text{for $|x|>0$ small}.\ee
\item[(3)] $D(u, v)\le 0$, and $D(u, v)<0$ if and only if there is a constant $c_0>0$ such that
\begin{equation}\label{eq8}u(x)+v(x)\geq c_0|x|^{-\frac{N-2}{2}},\quad\text{for $|x|>0$ small}.\ee
\end{itemize}
\end{theorem}

Our next results are concerned with singular solutions. A basic question is {\it whether semi-singular solutions exist or not}. It turns out the answer depends on the dimensions.

\begin{theorem}[Non-existence of semi-singular solutions for $N\geq 4$]\label{th2}Let  $N\ge 4$ and $\mu_1, \mu_2, \bb>0$. Then \eqref{eq0} has no semi-singular solutions.
\end{theorem}

\begin{theorem}[Existence of semi-singular solutions for $N=3$]\label{thm:N3-semi}
Let $N=3$ and $\mu_1,\mu_2,\beta>0$.
Then \eqref{eq0} admits a positive radially symmetric semi-singular solution $(u,v)$ satisfying
\begin{equation}\label{eq:asymptotics-main}
\lim_{|x|\to0}u(x)=A>0,
\qquad
\lim_{|x|\to0}|x|^{1/2}v(x)=(4\mu_2)^{-1/4}.
\end{equation}
Moreover, the Pohozaev invariant $
D(u,v)=-\frac{\pi}{6\sqrt{\mu_2}}<0.
$

Similarly, there also exists a positive radial semi-singular solution satisfying
\[
\lim_{|x|\to0}v(x)=B>0,
\qquad
\lim_{|x|\to0}|x|^{1/2}u(x)=(4\mu_1)^{-1/4}.
\]
\end{theorem}

Concerning both-singular solutions, we have the following sharp pointwise estimates.

\begin{theorem}\label{th20}Let  $N\ge 4$, $\mu_1, \mu_2, \bb>0$, and $(u, v)$ be a solution of \eqref{eq0} with $D(u, v)<0$. Then $(u,v)$ is both-singular, and there exist constants ${C}_2>{C}_1>0$ and $\alpha \in (0,1]$ such that

\begin{itemize}
\item[$(1)$]
For $N\geq 5$, 
\begin{equation}\label{asss}{C}_1|x|^{\frac{2-N}{2}}\le u(x), v(x)\le {C}_2|x|^{\frac{2-N}{2}},\quad\text{for $|x|>0$ small}.\end{equation}
\item[$(2)$]
For $N=4$, 
\begin{equation}\label{eq-alpha0}
        C_1|x|^{-\alpha}\le u(x),v(x)\le C_2|x|^{-1},\quad\text{for $|x|>0$ small}.
\end{equation}
Moreover, we can take $\alpha=1$ in \eqref{eq-alpha0} if $\beta>\max\{\mu_1,\mu_2\}$ or $\beta=\mu_1=\mu_2$.
\end{itemize}\end{theorem}

Remark that for the special case $N=4$ with $\beta=\mu_1=\mu_2$, system \eqref{eq0} is essentially of the form \eqref{eq05}, and $\alpha=1$ for this special case is already covered by \cite[Theorem 1.3]{COS}. Theorem \ref{th20} is new for all other cases.
In view of \eqref{eq8}, it is natural to ask whether one can also take $\alpha=1$ in \eqref{eq-alpha0} for $0<\beta\leq\max\{\mu_1,\mu_2\}$. Surprisingly, our next result shows that this can not hold for some both-singular solutions when $0<\beta<\max\{\mu_1,\mu_2\}$.

\begin{theorem}\label{th120}
Let  $N=4$ and $0<\beta<\max\{\mu_1,\mu_2\}$, say $\beta<\mu_1$ for example. Then \eqref{eq0} has a radially symmetric solution $(u,v)$ with $D(u,v)<0$ such that
$$\;\lim_{x\to 0} |x|u(x)=\frac{1}{\sqrt{\mu_1}}\quad\text{and }\;\;\lim_{x\to 0} |x|v(x)=0,$$
namely we can not take $\alpha=1$ in \eqref{eq-alpha0} for $v(x)$.
\end{theorem}
Theorem \ref{th120}  indicates that Theorem \ref{th20}-(2) is almost optimal, and also provides a novel phenomenon for the weakly coupled system \eqref{eq0} compared with the scalar equation \eqref{eq1-2}. 
 This provides an evidence that the system \eqref{eq0} is much more complicated than the scalar equation \eqref{eq1-2}.

Finally, we study the criterion of removable singularity.
Clearly, if the isolated singularity $0$ is removable for both $u$ and $v$, then by letting $r\to 0$ in \eqref{eq7}, we immediately obtain $D(u, v)=0$. Our final result is concerned with the removable singularity when $D(u, v)=0$ and $N\geq 4$. 

\begin{theorem}\label{th3}Let $N\geq 4$, $\mu_1, \mu_2, \bb>0$, and $(u, v)$ be a solution of \eqref{eq0} with $D(u, v)=0$. We further assume
$$ \beta>\min\{\mu_1,\mu_2\}\quad \text{or}\quad \beta=\mu_1=\mu_2,\quad\text{for }N=4.$$
Then the isolated singularity $0$ is removable for both $u$ and $v$, i.e., $u, v\in C^2(B_2)$.
\end{theorem}

Remark that for the special case $N=4$ with $\beta=\mu_1=\mu_2$, system \eqref{eq0} is essentially of the form \eqref{eq05}, and Theorem \ref{th3} for this special case is already covered by \cite[Theorem 1.5]{GKS}. Theorem \ref{th3} is new for all other cases.
In view of Theorem \ref{th3}, the case $N=3$ remains completely open. We should study this open problem elsewhere.

The asymptotic symmetry \eqref{eq00} can be proved by following the approach
of the ``measure theoretic" variation of the method of moving planes  from \cite{CGS,Li1996}. 
This asymptotic symmetry allows us to turn to the study of the following ODE
\begin{equation}\label{eq3-017}
\begin{cases}
\bar{\omega}_1''-\dd^2\bar{\omega}_1+\left(\mu_1\bar{\omega}_1^{2^\ast-1}+\bb \bar{\omega}_1^{\frac{2^\ast}{2}-1}\bar{\omega}_2^{\frac{2^\ast}{2}}\right)
(1+O(e^{-t}))=0,\\
\bar{\omega}_2''-\dd^2\bar{\omega}_2+\left(\mu_2\bar{\omega}_2^{2^\ast-1}+\bb \bar{\omega}_2^{\frac{2^\ast}{2}-1}\bar{\omega}_1^{\frac{2^\ast}{2}}\right)
(1+O(e^{-t}))=0,
\end{cases}
\end{equation}
for $t>0$ large.
Compared to the associated ODE of the scalar equation \eqref{eq1-2} studied in \cite{CGS}, the weakly coupled terms $\bar{\omega}_1^{\frac{2^\ast}{2}-1}\bar{\omega}_2^{\frac{2^\ast}{2}}$ and $\bar{\omega}_2^{\frac{2^\ast}{2}-1}\bar{\omega}_1^{\frac{2^\ast}{2}}$ of \eqref{eq3-017} brings many difficulties. 
Compared to the associated ODE \eqref{eq30-017}
of system \eqref{eq1-5} studied in  \cite{ChenLin1}, the error terms $O(e^{-t})$ of \eqref{eq3-017} brings additional difficulties. For example, for \eqref{eq30-017},
the functions
$$f_i(t):=-\frac12|v_i'(t)|^2+\frac{\delta^2}{2}v_i(t)^2-\frac{\mu_i}{2^\ast}v_i(t)^{2^\ast},\quad i=1,2$$
satisfy $f_i'(t)=\beta v_i^{\frac{2^\ast}{2}-1}v_{3-i}^{\frac{2^\ast}{2}} v_i'(t)$, namely the monotonicity of $f_i$ is exactly the same as the monotonicity of $v_i$. This property plays a crucial role in the ODE analysis of \eqref{eq30-017} in \cite{ChenLin1}. However, this property does not hold for \eqref{eq3-017} due to the error terms. Thus, many ideas of \cite{CGS, ChenLin1} can not work for \eqref{eq3-017}, which requires us to develop new ideas in the proofs of Theorems \ref{th2}, \ref{th20} and \ref{th3}. Here our techniques depend essentially on the relation between $\frac{2^\ast}{2}$ and $2$. In particular, our arguments work well for $N\geq 5$ due to $\frac{2^\ast}{2}<2$. Since $\frac{2^\ast}{2}=2$ for $N=4$, most arguments of $N\geq 5$ can not be applied to $N=4$, and we need to develop different ideas to deal with $N=4$. Unfortunately, since $\frac{2^\ast}{2}=3>2$ for $N=3$, our arguments of proving Theorem \ref{th3} for $N\geq 4$ fail for $N=3$, so the case $N=3$ requires different ideas and remains open.

The sturcture of the paper is organized as follows. In Section 2, we prove the asymptotic symmetry \eqref{eq00}, so we can turn to study the ODE \eqref{eq3-017}. By analysing this ODE, we complete the proofs of Theorems \ref{th1}, \ref{th2} and \ref{th20} in Section 3. 
In Section 4, we prove Theorems \ref{thm:N3-semi} and \ref{th120} by using the stable invariant manifold theory of ODEs. Finally, we prove Theorem \ref{th3} in Section 5.
Throughout the paper, we always use $C, c, C_0, c_0, C_1, c_1,\cdots$ to denote various positive constants that may be different in different places.

\section{Asymptotic symmetry}

Let $(u,v)$ be a positive solution of \eqref{eq0}. Recall $\mathbb{S}=\{\theta\in \RN\;|\; |\theta|=1\}$ and define the spheical averages
\be\label{eq2-1}\bar{u}(|x|)=\bar{u}(x):=\frac1{|\mathbb{S}|}\int_{\mathbb{S}}u(|x|\theta)\,d\theta,\quad \bar{v}(|x|)=\bar{v}(x):=\frac1{|\mathbb{S}|}\int_{\mathbb{S}}v(|x|\theta)\,d\theta.\ee
The main purpose of this section is to prove the following asymptotic symmetry.

\begin{theorem}\label{thm-as}
Let $(u,v)$ be a positive solution of \eqref{eq0}, and $(\bar {u}, \bar{v})$ be the spherical averages defined in \eqref{eq2-1}. Then
\begin{equation}\label{eq02-1}u(x)=(1+O(|x|))\bar{u}(|x|),\quad v(x)=(1+O(|x|))\bar{v}(|x|),\quad\text{for $|x|>0$ small}. \end{equation}
\end{theorem}

We will prove Theorem \ref{thm-as} by following the approach of the ``measure theoretic" variation of the method of moving planes  from \cite{CGS,Li1996}.
Write $$f(u, v)=\mu_1 u^{2^\ast-1}+\beta u^{\frac{2^\ast}{2}-1} v^{\frac{2^\ast}{2}},\quad g(u, v)=\mu_2 v^{2^\ast-1}+\beta v^{\frac{2^\ast}{2}-1}u^{\frac{2^\ast}{2}}$$ for convenience.

\bl\label{lemma2-1} We have $f, g\in L^1(B_1)$, $u, v\in L^{2^\ast-1}(B_1)$ and
$$-\Delta u=f(u, v),\quad-\Delta v=g(u, v)\quad\text{in $B_1$}$$
in the sense of distribution. Consequently, $u, v\in W^{1, q}(B_1)$ for any $1\leq q<\frac{N}{N-1}$.
\el

\begin{proof}
Let $\eta\in C^{\iy}(\R)$ such that $\eta'(t)\le 0$, $\eta(t)\equiv 1$ for $t\le 1$ and
$\eta(t)\equiv 0$ for $t\ge\frac32$. Define
$\eta_k(t):=\eta(t/k)$, then $\Phi_k(t):=\int_{0}^t\eta_k(s)ds\le 2k$.
For $0<\e\ll 1$ and $k\gg 1$, we have
{\allowdisplaybreaks
\begin{align*}
&\quad\int_{B_{\frac32}}f(u, v)\eta(|x|) (1-\eta_\e)(|x|)\eta_k(u) dx\\&=-\int_{B_{\frac32}}\eta(|x|) (1-\eta_\e)(|x|)\eta_k(u) \Delta u dx
=\int_{B_{\frac32}}\nabla u\nabla[\eta(|x|) (1-\eta_\e)(|x|)\eta_k(u)] dx\\
&=\int_{B_{\frac32}}\underbrace{|\nabla u|^2\eta(|x|) (1-\eta_\e)(|x|)\eta'_k(u)}_{\le 0} dx +\int_{B_{\frac32}\setminus B_1}\underbrace{\nabla u\nabla\eta(|x|) (1-\eta_\e)(|x|)\eta_k(u)}_\text{bounded} dx\\
&\quad-\int_{B_{\frac32\e}\setminus B_\e}\nabla u \nabla\eta_\e(|x|)\eta(|x|)\eta_k(u) dx\\
&\le C-\int_{B_{\frac32\e}\setminus B_\e}\nabla\eta_\e(|x|)\nabla\Phi_k(u) dx
=C+\int_{B_{\frac32\e}\setminus B_\e}\Delta\eta_\e(|x|)\Phi_k(u) dx\\
&\le C+Ck\e^{N-2}.
\end{align*}
}%
Letting $\e\to 0$, we obtain $\int_{B_{1}}f(u, v)\eta_k(u) dx\le C$. Then we let $k\to+\iy$ and obtain
$$\int_{B_{1}}f(u, v)dx\le C.$$
Similarly, $\int_{B_{1}}g(u, v)dx\le C$. Consequently, $u, v\in L^{2^\ast-1}(B_1)$. 

For any $\psi\in C_0^1(B_1)$, we have
{\allowdisplaybreaks
\begin{align*}
&\quad\int_{B_1}f(u, v)\psi(x)(1-\eta_\e(|x|))dx\\
&=-\int_{B_1}\Delta u\psi(1-\eta_\e)dx=
\int_{B_1}u\{-\Delta[\psi(1-\eta_\e)]\}dx\\
&=\int_{B_1}u(-\Delta\psi)(1-\eta_\e)dx+2\int_{B_1}u\nabla\psi\nabla\eta_\e dx
+\int_{B_1}u\psi\Delta\eta_\e dx,
\end{align*}
}%
which implies from $u\in L^{2^\ast-1}(B_1)$ that
{\allowdisplaybreaks
\begin{align*}
&\quad\left|\int_{B_1}(f(u, v)\psi+u\Delta\psi)(1-\eta_\e)dx\right|\\
&\le\left|\int_{B_1}\left(u\psi\Delta\eta_\e+2u\nabla\psi\nabla\eta_\e\right) dx\right|\le \frac{C}{\e^2}\int_{B_{\frac32\e}\setminus B_\e}udx\\
&\le \frac{C}{\e^2}\bigg(\int_{B_{\frac32\e}\setminus B_\e}u^{2^\ast-1}dx\bigg)^\frac{1}{2^\ast-1}\bigg(\int_{B_{\frac32\e}\setminus B_\e}1dx\bigg)^\frac{4}{N+2}\\&\leq C\e^{2\frac{N-2}{N+2}}\to 0,\qquad\text{as}\quad \e\to 0.
\end{align*}
}%
Since $f(u, v)\psi+u\Delta\psi\in L^1(B_1)$, we get
$$\int_{B_1}(f(u, v)\psi+u\Delta\psi)dx=0,$$
namely, $-\Delta u=f(u, v)$ in $B_1$ in the sense of distribution. Let
$$
\begin{cases}-\Delta u_1=f(u, v)\quad\text{in $B_1$},\\
u_1|_{\partial B_1}=0,\end{cases}\quad\begin{cases}-\Delta u_2=0\quad\text{in $B_1$},\\
u_2|_{\partial B_1}=u,\end{cases}
$$
then $u_2\in C(\overline{B_1})\cap C^2(B_1)$ and $u_1\in W_0^{1, q}(B_1)$ for any $1\leq q<\frac{N}{N-1}$. From this we conclude that $u=u_1+u_2\in W^{1, q}(B_1)$. Similarly,  $-\Delta v=g(u, v)$ in $B_1$ in the sense of distribution and $v\in W^{1,q}(B_1)$.
The proof is complete.
\end{proof}

Recall the spherical averages $(\bar {u}, \bar{v})$ defined in \eqref{eq2-1}, we have
$$
\begin{cases}
\begin{split}-\Delta \bar{u} = \overline{f(u, v)}=:\bar{f}(|x|)\\
-\Delta \bar{v} =\overline{g(u, v)}=:\bar{g}(|x|)\\
\end{split}  \quad   \text{in $B_2\setminus\{0\}$}.\end{cases}
$$

\begin{lemma}\label{lemma2-2} As $r\to 0$, 
$$r^{N-2}\bar{u}(r)\to 0,\quad r^{N-1}\bar{u}'(r)\to 0,\quad r^{N-2}\bar{v}(r)\to 0,\quad r^{N-1}\bar{v}'(r)\to 0.$$
\end{lemma}

\begin{proof}
By \eqref{eq2-1} and Lemma \ref{lemma2-1}, we have that for $0<r<1$, 
\begin{align}\label{eq2-55}r^{N-1}\bar{u}'(r)&=\frac{r^{N-1}}{|\mathbb{S}|}\int_{\mathbb{S}}\nabla u(r\theta)\cdot\theta d\theta=
\frac{1}{|\mathbb S|}\int_{\partial B_r}\frac{\partial u}{\partial \nu}d \sigma=\frac1{|\mathbb{S}|}\int_{B_r}\Delta u dx\nonumber\\
&=-\frac1{|\mathbb{S}|}\int_{B_r}f(u,v)dx,\end{align}
so
 $r^{N-1}\bar{u}'(r)=o(1)$ as $r\to 0$. Furthermore,
we have for $0<r<r_0<1$ that
$$\bar{u}(r_0)-\bar{u}(r)=-\frac1{|\mathbb{S}|}\int_r^{r_0}\frac{\int_{B_\rho}f dx}{\rho^{N-1}} d\rho,$$
which implies
$$r^{N-2}\bar{u}(r)=r^{N-2}\bar{u}(r_0)+\frac{r^{N-2}}{|\mathbb{S}|}\int_r^{r_0}\frac{\int_{B_\rho}f dx}{\rho^{N-1}} d\rho.$$
Note that
$$\int_r^{r_0}\frac{\int_{B_\rho}f dx}{\rho^{N-1}} d\rho\le \frac{r^{2-N}-r_0^{2-N}}{N-2}\int_{B_{r_0}}f dx,$$
which gives
$$\frac{r^{N-2}}{|\mathbb{S}|}\int_r^{r_0}\frac{\int_{B_\rho}f dx}{\rho^{N-1}} d\rho
\le\frac{1}{(N-2)|\mathbb{S}|}(1-r^{N-2}r_0^{2-N})\int_{B_{r_0}}f dx.$$
Thus,
$$0\leq\limsup_{r\to 0} r^{N-2}\bar{u}(r)\le\frac{1}{(N-2)|\mathbb{S}|}\int_{B_{r_0}}f dx$$
for any $0<r_0<1$. Letting $r_0\to 0$ we obtain $\lim_{r\to 0} r^{N-2}\bar{u}(r)=0$. The statements for $\bar v$ can be proved similarly.
\end{proof}

Define the Kelvin transform
$$U(x):=|x|^{2-N}u\left(\frac{x}{|x|^2}\right),\quad V(x):=|x|^{2-N}v\left(\frac{x}{|x|^2}\right),\quad\text{for }\;|x|>\frac12,$$
and again denote by $\overline{U}(x), \overline{V}(x)$ 
to be the spheical averages of $U(x), V(x)$ respectively.
Since the maximum principle holds for super harmonic functions with isolated singularities (cf.
\cite[Lemma 2.1]{CL1}), there exists a constant $A>0$ such that 
\be\label{eq2-05}u(x), v(x)\ge 4A>0,\quad \forall\,|x|\le \frac{3}{2},\ee so
$$U(x), V(x)\ge\frac{4A}{|x|^{N-2}},\quad\forall |x|\ge \frac{2}{3}.$$
Furthermore, it follows from Lemma \ref{lemma2-2} that
\begin{equation}\label{eq2-10}\overline{U}(x), \overline{V}(x)\to 0\quad\text{as $|x|\to +\iy$}.\end{equation}

For $x\in\RN$, $\la>0$ and any direction $\theta\in \mathbb{S}$, let $$x(\la, \theta ):=x+2(\la-x\cdot \theta)\theta$$ be the reflection point of $x$ about the hyperplane $$P(\la, \theta):=\{x\in\RN \;|\; x\cdot\theta=\la\},$$ and let $$\Sigma(\la,\theta):=\{x\in\RN \;|\; x\cdot\theta>\la\}$$ be the positive half space in the $\theta$ direction at $\la$.
Let $R(\theta)$ be the smallest nonnegative number such that
\begin{equation}\label{eq2-03}U(x)\le U(x(\la,\theta)),\; V(x)\le V(x(\la,\theta))\quad\text{for $x\in\Sigma(\la, \theta)$ and $x(\la,\theta)\not\in B_1$}\end{equation}
holds for all $\la\ge R(\theta)$ (At the moment $R(\theta)$ may be $+\iy$). We want to prove the existence of $M>0$ large such that $R(\theta)<M$ for all $\theta\in\mathbb S$; see Lemma \ref{lemma2-5} below.

\bl\label{lemma2-3}There exists $\dd_0\in (0,1)$ independent of $\theta\in\mathbb S$ such that, if $R(\theta)\ge 30$, then $U(z)+V(z)\ge 2A$ in the cylinder 
$$\Om_{\theta}:=\Big\{z=y+s\theta\;\Big|\; y\in B_{\dd_0}(y_\theta), s\in\mathbb{R}, 1<z\cdot\theta<2R(\theta)-1\Big\}$$ 
of radius ${\dd_0}$, where $y_\theta\in \mathbb S$ may depend on $\theta$.\el

\begin{proof} Take any $\theta\in\mathbb S$ such that $R(\theta)\geq 30$.
Without loss of generality, we may assume $\theta=e_1=(1,0,\cdots,0)$, and write $$x^\la=x(\la,e_1)=(2\la-x_1,x_2,\cdots,x_N),\quad\Sigma_\la=\Sigma(\la, e_1)=\{x\in\mathbb R^N \,|\, x_1>\la\}$$ for convenience. For fixed $\e\in (0,\frac1{100})$ small and $|x|\ge 1$, we define $x_\e:=(\e, 0,\cdots, 0)$ and
\be\label{eq2-06}U_\e(x):=|x|^{2-N}u\left(\frac{x}{|x|^2}+x_\e\right),\quad V_\e(x):=|x|^{2-N}v\left(\frac{x}{|x|^2}+x_\e\right).\ee
Remark that $\frac{x}{|x|^2}+x_\e=0$ implies $x=\bar{x}_\e:=(-\e^{-1}, 0,\cdots, 0)$. Therefore,
$$
\begin{cases}
\Delta U_\e+f(U_\e, V_\e)=0,\\
\Delta V_\e+g(U_\e, V_\e)=0,
\end{cases}\text{in $\RN\setminus D_\e$},\quad\text{where }\; D_\e:=B_1\cup\{\bar{x}_\e\}.
$$

First, we claim that there exists large $\la_0>0$ such that for any $\la\geq \la_0$, 
\be\label{movingplane}U_\e(x)\le U_\e(x^\la),\quad V_\e(x)\le V_\e(x^\la),\quad\text{for }\;x\in\Sigma_\la\;\text{and} \;x^\la\not\in D_\e.\ee

By \cite[Lemma 2.3]{CGS}, there exist large positive constants $\bar{\la}, R>|\bar{x}_\e|=\frac{1}{\e}$ such that for any $\la\ge \bar{\la}$,
\be\label{eq2-50}U_\e(x)< U_\e(x^\la),\quad V_\e(x)< V_\e(x^\la)\quad\text{for $x\in\Sigma_\la$ satisfying $|x^\la|>R$}.\ee
Since there is $C_0>0$ such that
$$\min_{|x|=1, R}U_\e\ge C_0>0,\quad \min_{|x|=1, R}V_\e\ge C_0>0,$$
the maximum principle gives
$$\min_{B_R\setminus D_\e}U_\e\ge C_0>0,\quad \min_{B_R\setminus D_\e}V_\e\ge C_0>0.$$
Clearly $U_\e(x), V_\e(x)\to 0$ as $|x|\to+\iy$ (At the moment we do not know whether this property holds for $U, V$, and this is the reason why we consider $U_\e, V_\e$ here), so there exists $\la_0>\bar{\la}$ such that $U_\e(x), V_\e(x)< C_0$ for $|x|\ge \la_0$. Consequently, for any $\la\ge \la_0$ and any $x\in \Sigma_\la$, we have $|x|\ge x_1>\la\ge \la_0$, so if $|x^\la|\le R$ and $x^\la\not\in D_\e$, then $U_\e(x^\la)\ge C_0>U_\e(x)$ and so does $V_\e$. Together with \eqref{eq2-50}, we see that
\eqref{movingplane} holds.

Let $R_\e$ be the smallest nonnegative number such that (\ref{movingplane}) holds for all $\la\ge R_\e$. 
Then by \eqref{eq2-03} and \eqref{movingplane}, it is easy to see that
\begin{equation}\label{eq2-07}\liminf_{\e\to 0}R_{\e}\geq R(e_1)\geq 30,\end{equation}
so we may assume $R_\e\ge 25$ for $\e>0$ small. Clearly,
$$U_\e(x)\le U_\e(x^{R_\e}),\;V_\e(x)\le V_\e(x^{R_\e})\quad\text{for $x\in\Sigma_{R_\e}$ and $x^{R_\e}\not\in D_\e$}.$$
By the definition of $R_\e$ and the maximum principle,
there are only two cases:
\begin{itemize}
\item[(1)] there exists $\tilde{x}_\e\in\Sigma_{R_{\e}}$ with $y_\e:=\tilde{x}_\e^{R_\e}\in\partial B_1$ such that either $U_\e(\tilde{x}_\e)=U_\e(y_\e)$ or $V_\e(\tilde{x}_\e)=V_\e(y_\e)$.

\item[(2)] there exists $x_k\in\Sigma_{R_\e}$ such that $x_k^{R_\e}\to\bar{x}_\e$ (the singular point) and either $U_\e(x_k)-U_\e(x_k^{R_\e})\to0$ or $V_\e(x_k)-V_\e(x_k^{R_\e})\to0$ as $k\to+\iy$.

\end{itemize}

We claim that (1) must hold. Indeed,
if (2) holds but (1) does not hold, we may assume, without loss of generality, that $U_\e(x_k)-U_\e(x_k^{R_\e})\to0$. Note that
$$-\Delta(U_\e(x^{R_\e})-U_\e(x))=f(U_\e(x^{R_\e}), V_\e(x^{R_\e}))-f(U_\e(x), V_\e(x))\ge 0$$
in $\Sigma_{R_\e}$ with $x^{R_\e}\not\in D_\e$. Fix $\dd>0$ sufficiently small. If there exists $x_0^{R_\e}\in\partial B_\dd(\bar{x}_\e)$ such that
$U_\e(x_0^{R_\e})-U_\e(x_0)=0$, then the strong maximum principle gives $U_\e(x^{R_\e})\equiv U_\e(x)$ and so (1) holds, a contradiction with our assumption. Hence there is a constant $C_1>0$ such that
$$\min_{x^{R_\e}\in \partial B_\dd(\bar{x}_\e)}(U_\e(x^{R_\e})-U_\e(x))\ge C_1>0.$$
Applying the maximum principle, we get
$$\min_{x^{R_\e}\in B_\dd(\bar{x}_\e)}(U_\e(x^{R_\e})-U_\e(x))\ge C_1>0,$$
a contradiction with $x_k^{R_\e}\to\bar{x}_\e$ and $U_\e(x_k)-U_\e(x_k^{R_\e})\to0$. 

Therefore, (1) holds, namely there exists $\tilde{x}_\e\in\Sigma_{R_\e}$ with $y_\e=\tilde{x}_\e^{R_\e}\in\partial B_1$ such that either $U_\e(\tilde{x}_\e)=U_\e(y_\e)$ or $V_\e(\tilde{x}_\e)=V_\e(y_\e)$.
Without loss of generality, we may assume $U_\e(\tilde{x}_\e)=U_\e(y_\e)$. It follows from $y_\e\in\partial B_1$, \eqref{eq2-06} and \eqref{eq2-05} that
$$|\tilde{x}_\e-(2R_{\e},0,\cdots,0)|=1,$$
$$U_\e(\tilde{x}_\e)=U_\e(y_\e)=u(y_\e+x_\e)\ge 4A.$$

Note that $B_2(\tilde{x}_\e)$ reflects into $B_{10}\setminus B_1$ with respect to the hyperplane $\{x : x_1=R_\e+3\}$. Then it follows from \eqref{movingplane} with $\lambda=R_\e+3$ and \eqref{eq2-05} that there exist $C_2, C_3>0$ independent of $\theta\in\mathbb S$ and small $\e$ such that
$$\sup_{B_2(\tilde{x}_\e)}U_\e\le \sup_{B_{10}\setminus B_1}U_\e\le C_2\sup_{B_{3/2}\setminus B_{1/12}}u\le C_3,$$
and the same holds for $V_\e$. Consequently,
$$\Delta (U_\e+V_\e)+C_\e(x)(U_\e+V_\e)=0\quad\text{in $B_2(\tilde{x}_\e)$},$$
where $C_\e(x)$ is bounded with a uniform bound independent of $\theta\in\mathbb S$ and small $\e$. Then by the gradient estimate, there is a constant $C>0$ independent of $\theta\in\mathbb S$ and small $\e$  such that
$$\sup_{B_1(\tilde{x}_\e)}|\nabla U_\e+\nabla V_\e|\le C.$$
Since $U_\e(\tilde{x}_\e)\ge 4A$, so there exists ${\dd_0}\in (0,1)$ independent of $\theta\in\mathbb S$ and small $\e$ such that
$$U_\e+V_\e\ge 2A \quad\text{in $B_{\dd_0}(\tilde{x}_\e)$}.$$
By the definition of $R_\e$, $U_\e+V_\e$ is monotone decreasing in the $e_1$ direction for $x\in\Sigma_{R_\e}=\{x : x_1>R_{\e}\}$. From here and $|\tilde{x}_\e|=|y_\e^{R_\e}|\ge 2R_\e-1$, we obtain
$$U_\e+V_\e\ge 2A \;\text{for}\; x\in \Big\{z=y+se_1\;\Big|\; y\in B_{\dd_0}(y_\e), R_\e<z_1=z\cdot e_1<2R_\e-1\Big\}.$$
Applying $(U_\e+V_\e)(x)\le (U_\e+V_\e)(x^{R_\e})$ for $x\in\Sigma_{R_\e}=\{x : x_1>R_{\e}\}$ and $x^{R_\e}\not\in D_\e$, we get
$$U_\e+V_\e\ge 2A \quad\text{for}\; x\in \Big\{z=y+se_1\;\Big|\; y\in B_{\dd_0}(y_\e), 1<z\cdot e_1<2R_\e-1\Big\}.$$
Up to a subsequence, $y_\e\to y_{e_1}\in\mathbb S$ as $\e\to 0$. Using \eqref{eq2-07}, we obtain $U+V\geq 2A$ for $x\in\Omega_{e_1}$. This completes the proof.
\end{proof}

\begin{lemma}\label{lemma2-4} For $t>0$, we define
$$E(t):=\{\theta\in \mathbb{S}\;|\; R(\theta)\geq t\}.$$
Then
$$|E(t)|\to 0\quad\text{as $t\to+\iy$}.$$
\el

\begin{proof}
Let $S(t):=\{x\in \partial B_t \;|\; U(x)+V(x)\ge 2A\}$, where $t> 30$.
We claim that $|S(t)|\ge \frac{1}{C}|E(t)|t^{N-1}$ for some constant $C>0$ independent of $t$. 

For $t> 30$ large, we pick disjoint balls of radius $3$ that centered at $t\theta$ for some
$\theta\in E(t)$ one by one. Let $N$ be the number of the balls we picked. Then
$N\ge \frac{1}{C_1}|E(t)|t^{N-1}$ for some $C_1>0$ independent of $t$. Let $t\theta\in tE(t)$ be the center of such a ball, then Lemma \ref{lemma2-3} shows that $U+V\ge 2A$ on a geodesic ball in $\partial B_t$ of radius $\dd_0$ that belongs to $B_3(t\theta)$ (Note the center of this geodesic ball is $y_\theta+s_\theta\theta$ with $s_\theta>0$ satisfying $|y_\theta+s_\theta\theta|=t$, which implies $|s_\theta-t|\leq 1$ and so $|y_\theta+s_\theta\theta-t\theta|\leq 2$). Hence, there are constants $C, C_2>0$ independent of $t$ such that
$$|S(t)|\ge C_2\dd_0^{N-1}N\ge \frac{1}{C}|E(t)|t^{N-1}.$$
Consequently, by writing $\overline{U}(x)=\overline{U}(|x|)$ and $\overline{V}(x)=\overline{V}(|x|)$, we have
{\allowdisplaybreaks
\begin{align*}
\overline{U}(t)+\overline{V}(t)&=\frac{1}{|\mathbb{S}|}\int_{\mathbb{S}}(U+V)(t\theta)d\theta
=\frac{1}{t^{N-1}|\mathbb{S}|}\int_{\partial B_t}(U+V)(t\theta)d\sigma\\
&\ge \frac{2A}{t^{N-1}|\mathbb{S}|}|S(t)|\ge \frac{2A}{C|\mathbb S|}|E(t)|.
\end{align*}
}%
The proof is complete by using \eqref{eq2-10} that $\overline{U}(t)+\overline{V}(t)\to 0$ as $t\to+\iy$.
\end{proof}

\bl\label{lemma2-5} Define
$$A(t):=\mathbb S\setminus E(t)=\{\theta\in \mathbb{S}\;|\; R(\theta)< t\}.$$
Then $U(x)+V(x)\to 0$ uniformly as $|x|\to +\iy$, and there is $M\geq 30$ large such that $A(M)=\mathbb{S}$, namely $R(\theta)<M$ for all $\theta\in \mathbb S$.
\el

\begin{proof}
For any given $x, y\in \RN$ with $|x|=6t$, $|y|=2t$ and $t>30$ large, define
\begin{align*}
&D_t:=\{z\in\partial B_{4t}\;|\; z\neq x(\la, \theta)\,\,\text{for any $\theta\in A(t)$ and any $\la>0$}\},\\
&Q_t:=\{z\in\partial B_{4t}\;|\; z\neq y(\la, \theta)\,\,\text{for any $\theta\in A(t)$ and any $\la>0$}\},\end{align*}
\begin{align*}
G_t:=\partial B_{4t}\setminus(D_t\cup Q_t).
\end{align*}
Then for any $z\in D_t$ (resp. $z\in Q_t$), we have $\frac{x-z}{|x-z|}\in \mathbb S\setminus A(t)= E(t)$ (resp. $\frac{z-y}{|z-y|}\in E(t)$), which implies $$|D_t|\le C|E(t)|t^{N-1},\quad |Q_t|\le C|E(t)|t^{N-1}$$ for some constant $C>0$ independent of $t$. Thus, it follows from Lemma \ref{lemma2-4} that 
$$|G_t|\ge t^{N-1}(4^{N-1}|\mathbb{S}|-2C|E(t)|)>0$$
for $t$ large, which implies that $G_t\neq \emptyset$ for $t$ large. Therefore, there exists $z\in\partial B_{4t}$ such that $z=x(\la_1, \theta_1)$ and $z=y(\la_2, \theta_2)$, where $\theta_i\in A(t)$ and $\lambda_i>0$ for $i=1, 2$. Then $R(\theta_i)<t$ and so
$$\la_1=\frac{|x|^2-|z|^2}{2|x-z|}\ge t>R(\theta_1),\quad \la_2=\frac{|z|^2-|y|^2}{2|y-z|}\ge {t}>R(\theta_2),$$
which, together with \eqref{eq2-03}, implies
$$U(x)\le U(z)\le U(y),\quad V(x)\le V(z)\le V(y).$$
That is,
$$\max_{\partial B_{6t}}U\le \min_{\partial B_{2t}}U\le \overline{U}(2t)\to 0,\quad \max_{\partial B_{6t}}V\le \min_{\partial B_{2t}}V\le \overline{V}(2t)\to 0,\;\text{as}\;t\to+\infty,$$
namely $U(x)+V(x)\to 0$ uniformly as $|x|\to +\iy$. Take $M\geq 30$ large such that $U(x)+V(x)<2A$ for any $|x|\geq M$. Then by Lemma \ref{lemma2-3}, we see that $R(\theta)<M$ for any $\theta\in\mathbb S$, 
so $A(M)=\mathbb{S}$.
\end{proof}

We are ready to prove Theorem \ref{thm-as}.

\begin{proof}[Proof of Theorem \ref{thm-as}] Let $R>M$.
For any $|y|=R$ and $|x|=R+2M$, we have $y=x(\la, \theta)$ with $\theta=\frac{x-y}{|x-y|}$ and $$\la=\frac{|x|^2-|y|^2}{2|x-y|}\ge M>R(\theta).$$ Consequently, $U(x)\le U(y)$ and $V(x)\le V(y)$, namely
$$\max_{\partial B_{R+2M}}U\le \min_{\partial B_R}U,\quad \max_{\partial B_{R+2M}}V\le \min_{\partial B_R}V.$$
On the other hand, the maximum principle implies
$$u(x)\geq \min_{\partial B_{1/R}}u\quad\text{for }|x|\leq 1/R,$$
from which and $U(x)=|x|^{2-N}u(x/|x|^2)$, we obtain
$$U(x)\ge \left(\frac{R}{|x|}\right)^{N-2}\min_{\partial B_R}U\quad\text{for $|x|\ge R$}.$$
Therefore,
$$\min_{\partial B_{R+2M}}U\ge\left(\frac{R}{R+2M}\right)^{N-2}\min_{\partial B_R}U\ge
\left[1+O\left((R+2M)^{-1}\right)\right]\max_{\partial B_{R+2M}}U$$
for $R$ large, that is,
$$\min_{\partial B_R}U\ge (1+O(R^{-1})) \max_{\partial B_R}U\quad\text{for $R$ large},$$
or equivalently, for $|x|=r>0$ small, 
$$u(x)\ge\min_{\partial B_r}u\ge (1+O(r))\max_{\partial B_r}u\ge (1+O(r))\bar{u}(|x|),$$
and $$\bar{u}(|x|)\ge \min_{\partial B_r}u\ge (1+O(r))u(x).$$This proves
$$u(x)=(1+O(|x|))\bar{u}(|x|),\quad\forall\, |x|\,\,\,\text{small}. $$
The same statement holds for $v$. The proof is complete.\end{proof}

\section{Sharp estimates for singular solutions}
\label{sec-3}

This section is devoted to the proofs of Theorems \ref{th1}, \ref{th2} and \ref{th20}.
First we prove the following pointwise upper bound, which is sharp for singular solutions.

\bl\label{lemma2-6}Let $\dd:=\frac{N-2}{2}$. Then there exists $C>0$ and $R_0\in (0,1)$ such that
$$u(x)\le C|x|^{-\dd},\quad v(x)\le C|x|^{-\dd},\quad\forall |x|\leq R_0.$$\el

\begin{proof} We follow the approach of \cite[Theorem 1.2]{CGS} to prove this lemma. Let $\eta(x)=|x|^2$ and $0<r<2$. By Lemma \ref{lemma2-1}, we have
{\allowdisplaybreaks
\begin{align*}
\int_{B_r}(u\Delta \eta+\eta f(u, v))dx&=\int_{B_r}(u\Delta \eta-\eta \Delta u)dx=\int_{\partial B_r}(2ru-r^2 \frac{\partial u}{\partial\nu})d\sigma\\
&=2r\int_{\partial B_r}ud\sigma-r^2\int_{B_r}\Delta udx\\
&=2r\int_{\partial B_r}ud\sigma+r^2\int_{B_r}f(u,v)dx,
\end{align*}
}%
so
\begin{equation}\label{eq3-1}\int_{B_r}(r^2-|x|^2)f(u, v)dx=2N\int_{B_r}udx-2r\int_{\partial B_r}ud\sigma< 2N\int_{B_r}udx.\end{equation}
Note from \eqref{eq2-55} that
\begin{align}\label{eq02-8}\bar{u}'(r)=
\frac{-1}{|\partial B_r|}\int_{B_r}f(u, v)dx<0,
\end{align}
namely $\bar{u}(r)$ is decreasing with respect to $r>0$. Applying the H\"{o}lder inequality to $\bar{u}(\rho)=\frac{1}{|\mathbb{S}|}\int_{\mathbb{S}}u(\rho \theta)d\theta$, we get
$$\bar{u}(\rho)^{2^\ast-1}\le\frac{1}{|\mathbb{S}|}\int_{\mathbb{S}}u(\rho \theta)^{2^\ast-1}d\theta,$$
so
{\allowdisplaybreaks
\begin{align*}
\frac{1}{|\mathbb{S}|}\int_{B_r}(r^2-|x|^2)f(u, v)dx\ge&\frac{\mu_1}{|\mathbb{S}|}\int_{B_r}(r^2-|x|^2)u^{2^\ast-1}dx\\
=&\frac{\mu_1}{|\mathbb{S}|}\int_0^r(r^2-\rho^2)\rho^{N-1}
\int_{\mathbb{S}}u(\rho\theta)^{2^\ast-1}d\theta d\rho\\
\ge &\mu_1\int_0^r (r^2-\rho^2)\rho^{N-1}\bar{u}(\rho)^{2^\ast-1}d\rho\\
\ge&\mu_1\bar{u}(r)^{2^\ast-1}\int_0^r (r^2-\rho^2)\rho^{N-1}d\rho\\
=& \frac{2\mu_1}{N(N+2)}r^{N+2}\bar{u}(r)^{2^\ast-1}.
\end{align*}
}%
Inserting this estimate into \eqref{eq3-1} leads to
\begin{align*}
r^{N+2}\bar{u}(r)^{2^\ast-1}&\le C\int_{B_r}u dx=C\int_0^r\rho^{N-1}\int_{\mathbb{S}}u(\rho\theta)d\theta d\rho\\
&= C|\mathbb S|\int_0^r\rho^{N-1}\bar{u}(\rho)d\rho\le Cr^{\frac{4N}{N+2}}
\left(\int_0^r\rho^{N-1}\bar{u}(\rho)^{2^\ast-1}d\rho\right)^{\frac{1}{2^\ast-1}}.
\end{align*}
This implies
$$r^{N-1}\bar{u}(r)^{2^\ast-1}\le Cr^{\frac{N-6}{N+2}}
\left(\int_0^R\rho^{N-1}\bar{u}(\rho)^{2^\ast-1}d\rho\right)
^{\frac{1}{2^\ast-1}},\quad\forall\,0<r\le R,$$
and then
\begin{align*}
\int_0^R r^{N-1}\bar{u}(r)^{2^\ast-1}dr&\le C\left(\int_0^R\rho^{N-1}\bar{u}(\rho)^{2^\ast-1}d\rho\right)
^{\frac{1}{2^\ast-1}}\int_0^R r^{\frac{N-6}{N+2}}dr\\
&\le C\left(\int_0^R\rho^{N-1}\bar{u}(\rho)^{2^\ast-1}d\rho\right)
^{\frac{1}{2^\ast-1}} R^{\frac{2N-4}{N+2}}.
\end{align*}
That is,
$$\frac1N R^N\bar{u}(R)^{2^\ast-1}\le \int_0^R r^{N-1}\bar{u}(r)^{2^\ast-1}dr\le C R^{\frac{N-2}{2}},$$
so $\bar{u}(R)\le C R^{-\frac{N-2}{2}}$. Similarly,  $\bar{v}(R)\le C R^{-\frac{N-2}{2}}$. The proof is complete by using Theorem \ref{thm-as}.
\end{proof}

Now similarly as \cite{CGS}, we use Fowler's change of variables.
Let $T_0:=-\ln R_0>0$, $r=e^{-t}$ and
define 
\begin{equation}\label{eq3-omega}\omega_1(t, \theta):=e^{-\dd t}u(e^{-t}\theta),\quad \omega_2(t, \theta):=e^{-\dd t}v(e^{-t}\theta),\quad \forall t\geq T_0,\end{equation}
$$\bar{\omega}_1(t):=e^{-\dd t}\bar{u}(e^{-t}),\quad \bar{\omega}_2(t):=e^{-\dd t}\bar{v}(e^{-t}),\quad \forall t\geq T_0.$$
Clearly,
\begin{align*}
\bar{\omega}_1(t)=e^{-\dd t}\bar{u}(e^{-t})=\frac{e^{-\dd t}}{|\mathbb{S}|}\int_{\mathbb{S}}u(e^{-t}\theta)d \theta
=\frac{1}{|\mathbb{S}|}\int_{\mathbb{S}}\omega_1(t, \theta)d \theta,
\end{align*}
and similarly,
$$\bar{\omega}_2(t)=e^{-\dd t}\bar{v}(e^{-t})=\frac{1}{|\mathbb{S}|}\int_{\mathbb{S}}\omega_2(t, \theta)d \theta.$$
Then
$$\bar u'(r)=-r^{-\delta-1}(\bar\omega_1'(t)+\delta\bar\omega_1(t)),$$
which together with \eqref{eq02-8} implies
\begin{equation}\label{eq02-9}
\bar\omega_1'(t)+\delta\bar\omega_1(t)>0,\quad\text{and similarly,}\quad
\bar\omega_2'(t)+\delta\bar\omega_2(t)>0.
\end{equation}
Furthermore, Lemma \ref{lemma2-6} implies
\begin{equation}\label{eq3-2}
\omega_i(t,\theta)\leq C_2, \quad \bar\omega_i(t)\leq C_2, \quad \forall t\geq T_0,\;\;i=1,2.
\end{equation}
\bl\label{lemma2-7} For $i=1,2$ and $t\geq T_0$ large, we have
\begin{equation}\label{eq3-3-1}
|\omega_i(t,\theta)-\bar{\omega}_i(t)|=\bar{\omega}_i(t)O(e^{-t})=O(e^{-t}),
\end{equation}
\begin{equation}\label{eq3-3-2}
|\bar{\omega}_i'(t)|=O(\bar{\omega}_1(t)+\bar{\omega}_2(t))=O(1),
\end{equation}
\begin{align}\label{eq3-3-3}
\left|\frac{\partial}{\partial t}(\omega_i(t,\theta)-\bar{\omega}_i(t))\right|=(\bar{\omega}_1(t)+\bar{\omega}_2(t))O(e^{-t})=O(e^{-t}),\end{align}
\begin{align}\label{eq3-3-4}
\big|\nabla_{\theta}\omega_i(t,\theta)\big|=(\bar{\omega}_1(t)+\bar{\omega}_2(t))O(e^{-t})=O(e^{-t}).
\end{align}
\el

\begin{proof} Clearly \eqref{eq3-3-1} follows directly from Theorem \ref{thm-as} and \eqref{eq3-omega}-\eqref{eq3-2}.

We follow the approach of \cite[Lemma 7.1]{CGS} to prove \eqref{eq3-3-2}-\eqref{eq3-3-4}.
Note from Theorem \ref{thm-as} that for $|x|>0$ small, 
\begin{align*}-\Delta u(x)&=\mu_1 u^{2^\ast-1}(x)+\bb u^{\frac{2^\ast}{2}-1}v^{\frac{2^\ast}{2}}(x)\\
&=\left[\mu_1 \bar{u}^{2^\ast-1}(x)+\bb \bar{u}^{\frac{2^\ast}{2}-1}\bar{v}^{\frac{2^\ast}{2}}(x)\right](1+O(|x|)),
\end{align*}
\begin{align}\label{eq3-16}-\Delta \bar{u}(x)&=\overline{\mu_1 u^{2^\ast-1}(x)+\bb u^{\frac{2^\ast}{2}-1}v^{\frac{2^\ast}{2}}(x)}\\
&=\left[\mu_1 \bar{u}^{2^\ast-1}(x)+\bb \bar{u}^{\frac{2^\ast}{2}-1}\bar{v}^{\frac{2^\ast}{2}}(x)\right](1+O(|x|)),\nonumber
\end{align}
which implies
$$-\Delta(u-\bar{u})(x)=\left[\mu_1 \bar{u}^{2^\ast-1}(x)+\bb \bar{u}^{\frac{2^\ast}{2}-1}\bar{v}^{\frac{2^\ast}{2}}(x)\right]O(|x|).$$
Then by the gradient estimate (cf. \cite[Corollary 2.2.7]{Jost}), $u(x)-\bar{u}(|x|)=\bar{u}(|x|)O(|x|)$ and Lemma \ref{lemma2-6}, there are constants $C>0$ independent of $r$ such that for $r>0$ small,
{\allowdisplaybreaks
\begin{align}\label{eq3-7}
&\quad|\nabla(u-\bar{u})(x)|\big|_{|x|=r}\nonumber\\
&\le C\left[\frac{1}{r}\sup_{\frac{2r}{3}\le|x|\le\frac{3r}{2}}|u-\bar{u}|
+r\sup_{\frac{2r}{3}\le|x|\le\frac{3r}{2}}(\mu_1 \bar{u}^{2^\ast-1}+\bb \bar{u}^{\frac{2^\ast}{2}-1}\bar{v}^{\frac{2^\ast}{2}})(r)O(r)\right]\nonumber\\
&\le C\sup_{\frac{2r}{3}\le|x|\le\frac{3r}{2}}\bar{u}+C r^2\sup_{\frac{2r}{3}\le|x|\le\frac{3r}{2}}(\bar{u}+\bar{v})^{2^\ast-1}\nonumber\\
&\le C\sup_{\frac{2r}{3}\le|x|\le\frac{3r}{2}} (\bar{u}+\bar{v}),
\end{align}
}%
and similarly, 
\be\label{eq03-9}|\bar{u}'(r)|\leq \frac{C}{r}\sup_{\frac{2r}{3}\le|x|\le\frac{3r}{2}} (\bar{u}+\bar{v}).\ee

On the other hand, for $|x|\in (\frac12, 2)$ and $r>0$ small, we let
$$
U_r(x):=r^{\frac{N-2}{2}}u(rx)\leq C|x|^{\frac{2-N}{2}},\quad V_r(x):=r^{\frac{N-2}{2}}v(rx)\leq C|x|^{\frac{2-N}{2}}.
$$
Then
$$
\begin{cases}
-\Delta U_r(x)=\mu_1 U_r^{2^\ast-1}(x)+\bb U_r^{\frac{2^\ast}{2}-1}V_r^{\frac{2^\ast}{2}}(x),\\
-\Delta V_r(x)=\mu_2 V_r^{2^\ast-1}(x)+\bb V_r^{\frac{2^\ast}{2}-1}U_r^{\frac{2^\ast}{2}}(x).
\end{cases}
$$
Thus, we can write
$-\Delta(U_r+V_r)(x)=C(x)(U_r+V_r)(x)$, where
\begin{align*}
C(x)\le C\left(U_r+V_r\right)^{2^\ast-2}(x)\le C|x|^{-2}\le 4C,
\end{align*}
where we have used $\frac12<|x|<2$. Therefore, the Harnack inequality gives
$$\sup_{\frac{2}{3}\le|x|\le\frac{3}{2}}(U_r+V_r)\le C\inf_{\frac{2}{3}\le|x|\le\frac{3}{2}}(U_r+V_r),$$
where $C>0$ is independent of $r$. Consequently,
$$\sup_{\frac{2r}{3}\le|x|\le\frac{3r}{2}}(u+v)\le C\inf_{\frac{2r}{3}\le|x|\le\frac{3r}{2}}(u+v),\quad\forall\,r>0\;\text{small}.$$
Inserting this into \eqref{eq3-7}-\eqref{eq03-9}, we get that for $r>0$ small,
$$|\nabla(u-\bar{u})|\big|_{|x|=r}\le C\inf_{\frac{2r}{3}\le|x|\le\frac{3r}{2}}(u+v)\le
C(\bar{u}+\bar{v})(r),$$
$$|\bar u'(r)|\leq \frac{C}{r}(\bar{u}+\bar{v})(r).$$
Thus, for $r>0$ small,
$$\left|\frac{\partial}{\partial r}(u-\bar{u})(r\theta)\right|\le C(\bar{u}+\bar{v})(r),
\quad \big|\nabla_\theta u(r\theta)\big|\le Cr(\bar{u}+\bar{v})(r), $$
and consequently, 
$$|\bar\omega_1'(t)|\leq\delta r^{\delta}\bar u(r)+r^{\frac{N}2}|\bar u'(r)|\leq Cr^{\delta}(\bar{u}+\bar{v})(r)= C(\bar{\omega}_1(t)+\bar{\omega}_2(t)),$$
\begin{align*}
\left|\frac{\partial}{\partial t}(\omega_1(t,\theta)-\bar{\omega}_1(t))\right|
&\le\dd r^{\dd}|u-\bar{u}|+r^{\frac{N}{2}}\left|\frac{\partial}{\partial r}(u-\bar{u})(r\theta)\right|\\
&\le Cr^{\frac{N}{2}}(\bar{u}+\bar{v})(r)=C(\bar{\omega}_1(t)+\bar{\omega}_2(t))e^{-t},
\end{align*}
$$|\nabla_\theta \omega_1(t,\theta)|= e^{-\dd t}|\nabla_\theta u(r\theta)|\le C(\bar{\omega}_1(t)+\bar{\omega}_2(t))e^{-t}.$$
This proves \eqref{eq3-3-2}-\eqref{eq3-3-4} for $i=1$. The case $i=2$ can be proved similarly.
This completes the proof.
\end{proof}

It follows from \eqref{eq0} that $(\omega_1, \omega_2)$ satisfies the following system
\begin{equation}\label{eq3-10}
\begin{cases}
(\omega_1)_{tt}+\Delta_\theta \omega_1-\dd^2\omega_1+\mu_1\omega_1^{2^\ast-1}+\bb \omega_1^{\frac{2^\ast}{2}-1}\omega_2^{\frac{2^\ast}{2}}=0,\\
(\omega_2)_{tt}+\Delta_\theta \omega_2-\dd^2\omega_2+\mu_2\omega_2^{2^\ast-1}+\bb \omega_2^{\frac{2^\ast}{2}-1}\omega_1^{\frac{2^\ast}{2}}=0.
\end{cases}
\end{equation}
Multiplying the first equation by $(\omega_1)_t$, the second equation by $(\omega_2)_t$, and integrating over $\mathbb{S}\times [t, s]$, we easily obtain
\begin{align}\label{eq3-11}\int_{\mathbb{S}}&\Big[\frac{1}{2}((\omega_1)_t^2
+(\omega_2)_t^2-\dd^2\omega_1^2-\dd^2\omega_2^2)
+\frac{1}{2^\ast}(\mu_1\omega_1^{2^\ast}\\
&+2\bb \omega_1^\frac{2^\ast}{2}\omega_2^\frac{2^\ast}{2}+\mu_2\omega_2^{2^\ast})
-|\nabla_\theta \omega_1|^2-|\nabla_\theta \omega_2|^2\Big]\bigg|_t^s=0.\nonumber\end{align}
Define
\begin{align}\label{eq3-12}D(t):=&\frac{1}{2}((\bar{\omega}_1')^2
+(\bar{\omega}_2')^2-\dd^2\bar{\omega}_1^2-\dd^2\bar{\omega}_2^2)(t)\\
&+\frac{1}{2^\ast}(\mu_1\bar{\omega}_1^{2^\ast}
+2\bb \bar{\omega}_1^\frac{2^\ast}{2}\bar{\omega}_2^\frac{2^\ast}{2}
+\mu_2\bar{\omega}_2^{2^\ast})(t),\quad t\geq T_0,\nonumber\end{align}
Then by inserting \eqref{eq3-2}-\eqref{eq3-3-4} into \eqref{eq3-11}, we obtain 
\begin{align}\label{eq3-13} D(s)-D(t)
=(\bar{\omega}_1^2+\bar{\omega}_2^2)O(e^{-\tau})\Big|_t^s\quad\text{for $s>t>T_0$ large},\end{align}
namely $D(s)-D(t)\to 0$ as $s>t\to +\iy$. Therefore,
\begin{align}\label{eq3-14}D_\iy:=\lim_{t\to+\iy}D(t)\in\mathbb R.\end{align}
Then letting $s\to+\infty$ in \eqref{eq3-13} gives
\begin{align}\label{eq3-15}D(t)=D_\iy+(\bar{\omega}_1(t)^2+\bar{\omega}_2(t)^2)O(e^{-t})\quad\text{for $t>T_0$ large}.\end{align}

\begin{lemma}\label{lemma03-3}
Recalling the constant $D(u,v)$ defined in \eqref{eq6}-\eqref{eq7}, we have $$D(u, v)=|\mathbb S|D_\iy\leq 0.$$
Furthermore,  
\be\label{eq3-22}D_\iy<0\quad\Longleftrightarrow\quad\liminf_{t\to+\infty}(\bar{\omega}_1(t)+\bar{\omega}_2(t))>0.\ee
In this case, there is $C>0$ such that
\be\label{eq3-23}u(x)+v(x)\geq C|x|^{-\delta},\quad\text{for $|x|>0$ small},\ee
namely at least one of $\{u(x), v(x)\}$ is singular at $x=0$.
\end{lemma}

\begin{proof}
By \eqref{eq6}-\eqref{eq7}, \eqref{eq3-11}-\eqref{eq3-13} and $r=e^{-t}$, we have
{\allowdisplaybreaks
\begin{align*}
D(u,v)\equiv D(r; u,v)=&\int_{\mathbb{S}}\Big[\frac{1}{2}((\omega_1)_t^2
+(\omega_2)_t^2-\dd^2\omega_1^2-\dd^2\omega_2^2)
+\frac{1}{2^\ast}(\mu_1\omega_1^{2^\ast}\\
&+2\bb \omega_1^\frac{2^\ast}{2}\omega_2^\frac{2^\ast}{2}+\mu_2\omega_2^{2^\ast})
-|\nabla_\theta \omega_1|^2-|\nabla_\theta \omega_2|^2\Big]d\theta\\
=&|\mathbb S|D(t)+(\bar{\omega}_1^2(t)+\bar{\omega}_2^2(t))O(e^{-t}).
\end{align*}}
Letting $t\to+\infty$, we immediately get $D(u, v)=|\mathbb S|D_\iy$.

It suffices to prove $D_\infty\leq 0$.
Note from \eqref{eq3-16} that for $|x|>0$ small,
$$
\begin{cases}
-\Delta \bar{u}(x)=\left[\mu_1 \bar{u}^{2^\ast-1}(x)+\bb \bar{u}^{\frac{2^\ast}{2}-1}\bar{v}^{\frac{2^\ast}{2}}(x)\right](1+O(|x|)),\\
-\Delta \bar{v}(x)=\left[\mu_2 \bar{v}^{2^\ast-1}(x)+\bb \bar{v}^{\frac{2^\ast}{2}-1}\bar{u}^{\frac{2^\ast}{2}}(x)\right](1+O(|x|)),
\end{cases}
$$
so for $t>T_0$ large,
\begin{equation}\label{eq3-17}
\begin{cases}
\bar{\omega}_1''-\dd^2\bar{\omega}_1+\left(\mu_1\bar{\omega}_1^{2^\ast-1}+\bb \bar{\omega}_1^{\frac{2^\ast}{2}-1}\bar{\omega}_2^{\frac{2^\ast}{2}}\right)
(1+O(e^{-t}))=0,\\
\bar{\omega}_2''-\dd^2\bar{\omega}_2+\left(\mu_2\bar{\omega}_2^{2^\ast-1}+\bb \bar{\omega}_2^{\frac{2^\ast}{2}-1}\bar{\omega}_1^{\frac{2^\ast}{2}}\right)
(1+O(e^{-t}))=0.
\end{cases}
\end{equation}
Take a sequence $t_n\uparrow+\iy$. By \eqref{eq3-2} and standard elliptic estimates, 
up to a subsequence,
$\bar\omega_1(\cdot+t_n)\to v_1\ge 0$ and $\bar\omega_2(\cdot+t_n)\to v_2\ge 0$ uniformly in $C^2_{loc}(\R)$, and
\be\label{eq3-9}
\begin{cases}
\begin{split}v_{1}''-\dd^2v_{1}+\mu_1v_1^{2^\ast-1}+\bb v_1^{\frac{2^\ast}{2}-1}v_2^{\frac{2^\ast}{2}}=0\\
v_{2}''-\dd^2v_{2}+\mu_2v_2^{2^\ast-1}+\bb v_2^{\frac{2^\ast}{2}-1}v_1^{\frac{2^\ast}{2}}=0\\
\end{split}  \quad   t\in\R.\end{cases}\ee
Clearly
\begin{align}\label{eq003-2}
K(t):=&\frac{1}{2}(|v_1'|^2+|v_2'|^2-\dd^2v_1^2-\dd^2v_2^2)(t)\nonumber\\
&+\frac{1}{2^\ast}(\mu_1v_1^{2^\ast}+2\bb v_1^{\frac{2^\ast}{2}}v_2^{\frac{2^\ast}{2}}+\mu_2v_2^{2^\ast})(t)\equiv K
\end{align}
is a contant independent of $t$. As mentioned in Section 1, 
the system \eqref{eq3-9} was already studied in \cite{ChenLin1}, and it was proved in \cite[Theorem 1.1]{ChenLin1} that $K\leq 0$. Then
$$D_\infty=\lim_{t_n\to+\infty}D(t_n)=K(0)=K\leq 0.$$

Finally, we prove \eqref{eq3-22}. First, suppose $D_\infty<0$. Assume by contradiction that there exists $t_n\uparrow+\iy$ such that $\bar\omega_1(t_n)+\bar\omega_2(t_n)\to 0$. Then up to a subsequence,
$\bar\omega_1(\cdot+t_n)\to v_1\ge 0$ and $\bar\omega_2(\cdot+t_n)\to v_2\ge 0$ uniformly in $C^2_{loc}(\R)$, where $(v_1, v_2)$ is a solution of \eqref{eq3-9} satisfying $v_1(0)=v_2(0)=0$. Consequently, $v_1=v_2\equiv 0$ and so $D_\iy=K=0$, a contradiction. Thus,
\be\label{eq3-25}\liminf_{t\to+\infty}(\bar{\omega}_1(t)+\bar{\omega}_2(t))>0,\ee
or equivalently, there is a constant $C_1>0$ such that
\be\label{eq3-26}\bar{\omega}_1(t)+\bar{\omega}_2(t)\geq C_1,\quad\text{for $t>T_0$ large},\ee
which implies
$$\bar u(x)+\bar v(x)\geq C_1|x|^{-\delta},\quad\text{for $|x|>0$ small}.$$
Applying Theorem \ref{thm-as}, we obtain \eqref{eq3-23}.

Conversely, suppose \eqref{eq3-25} or equivalently, \eqref{eq3-26} holds. Again take a sequence $t_n\uparrow+\iy$.
Up to a subsequence,
$\bar\omega_1(\cdot+t_n)\to v_1\ge 0$ and $\bar\omega_2(\cdot+t_n)\to v_2\ge 0$ uniformly in $C^2_{loc}(\R)$, where 
$(v_1, v_2)$ is a solution of \eqref{eq3-9} satisfying $v_1(t)+v_2(t)\geq C_1>0$ for all $t$. Then it follows from \cite[Theorem 1.1]{ChenLin1} that $K<0$ and so $D_\infty=K<0$.  The proof is complete.
\end{proof}

\begin{proof}[Proof of Theorem \ref{th1}]
Clearly Theorem \ref{th1} follows directly from Lemmas \ref{lemma2-6}, \ref{lemma03-3} and Theorem \ref{thm-as}.
\end{proof}

Now we turn to the proof of Theorem \ref{th2}.
First, note from \eqref{eq3-22} that
\begin{align}\label{eq3-22-1}D_\iy=0\quad&\Longleftrightarrow\quad\liminf_{t\to+\infty}(\bar{\omega}_1(t)+\bar{\omega}_2(t))=0\\
&\Longrightarrow\quad\liminf_{t\to+\infty}\bar{\omega}_1(t)=\liminf_{t\to+\infty}\bar{\omega}_2(t)=0.\nonumber\end{align}

\begin{lemma}\label{lemma3-1}
The singularity $0$ is removable for both $u$ and $v$ if and only if
\be\label{eq-3-31}
        W(t):=\bar\omega_1(t)+\bar\omega_2(t)\to 0,
        \quad \text{as }\; t\to+\infty .
\ee
In this case, it follows from \eqref{eq3-22-1} that $D_\iy=0$.
\end{lemma}

\begin{proof}
First we prove the necessary part. Suppose $u(x)\leq C$ and $v(x)\leq C$ for $|x|>0$ small, then so do $\bar u(x)$ and $\bar v(x)$, which implies \eqref{eq-3-31} by using $\bar{\omega}_1(t)=e^{-\dd t}\bar{u}(e^{-t})$ and $\bar{\omega}_2(t)=e^{-\dd t}\bar{v}(e^{-t})$.

Conversely, we suppose that \eqref{eq-3-31} holds.
Note from \eqref{eq3-2}  and \eqref{eq3-17} that
\begin{equation}\label{eq-W}-C_1W^{2^\ast-1}\ge W''-\dd^2W\ge-C_2W^{2^\ast-1},\quad\text{for $t> T_0$ large}.\end{equation}
Then by \eqref{eq-3-31}, we have $W''(t)>0$ for $t>T_0$ large and so $W'(t)\uparrow 0$ as $t\to +\iy$. That is, $W'(t)<0$ for $t>T_0$ large and $W(t)\downarrow 0$ as $t\to+\iy$.

Define
$y(t):=W'(t)+\dd W(t),$ then 
$$y'(t)-\dd y(t)=W(t)''-\dd^2 W(t)\ge -C_2 W(t)^{2^\ast-1},\quad\text{for $t> T_0$ large},$$
and so
$$\frac{d}{dt}(e^{-\delta t}y(t))\ge -C_2e^{-\dd t}W(t)^{2^\ast-1},\quad\text{for $t> T_0$ large}.$$
Since it follows from \eqref{eq02-9}-\eqref{eq3-2} and \eqref{eq3-3-2} that $y(t)>0$ is bounded for $t>T_0$ large, we have that for $t>T_0$ large,
\begin{align*}
0-e^{-\delta t}y(t)&\geq -C_2\int_{t}^{+\infty}e^{-\dd s}W(s)^{2^\ast-1}ds\\
&>-C_2W(t)^{2^\ast-1}\int_{t}^{+\infty}e^{-\dd s}ds=-\frac{C_2}{\dd}e^{-\dd t}W(t)^{2^\ast-1},
\end{align*}
so
$$W'(t)+\dd W(t)=y(t)<\frac{C_2}{\dd}W(t)^{2^\ast-1}.$$
This implies the existence of $T_1>T_0$ large such that $(e^{\dd t}W(t))^{2-2^\ast}-\frac{C_2}{\dd^2}e^{(2-2^\ast)\dd t}$ is strictly increasing for $t>T_1$. Assume by contradiction that $e^{\dd t_n}W(t_n)\to+\infty$ for a sequence $t_n\uparrow +\infty$, then we obtain
$$(e^{\dd t}W(t))^{2-2^\ast}-\frac{C_2}{\dd^2}e^{(2-2^\ast)\dd t}<0,\quad\text{for }\;t>T_1,$$ 
or equivalently, $W(t)\geq (\delta^2/C_2)^{\frac{1}{2^\ast-2}}$ for any $t>T_1$, a contradiction with
 $W(t)\downarrow 0$ as $t\uparrow+\iy$. Therefore, $e^{\dd t}W(t)\le C$ for any $t>T_1$ large. That is, $\bar u(r)+\bar v(r)\le C$ for $r>0$ small. Together with Theorem \ref{thm-as}, we obtain $u(x)+v(x)\leq C$ for $|x|>0$ small, so $u, v\in C^2(B_1)$ by using the standard elliptic regularity theory, namely the singularity $0$ is removable for both $u$ and $v$. This completes the proof.
\end{proof}

\bl\label{lemma2-8}If $D_\iy=0$, then $(u,v)$ can not be semi-singular.\el

\begin{proof} 
Since $D_{\iy}=0$, we believe that the singularity $0$ is removable for both $u$ and $v$ (We will prove this assertion under some conditions in the next section). Thus, here we assume that $u\le C$ for $|x|>0$ small, and we want to prove that $v$ is also bounded for $|x|>0$ small, so $(u,v )$ is not semi-singular.

Clearly $\bar{u}(|x|)\le C$ for $|x|>0$ small and $\bar{\omega}_1(t)\le Ce^{-\dd t}$ for $t$ large.
Note from \eqref{eq3-22-1} that $$\liminf_{t\to+\iy}\bar{\omega}_2=0.$$ Assume by contradiction that
$$\limsup_{t\to+\iy}\bar{\omega}_2>0.$$
Denote $W(t)=\bar{\omega}_1(t)+\bar{\omega}_2(t)$ as in Lemma \ref{lemma3-1}, then
$$\liminf_{t\to+\iy}W(t)=0,\quad \limsup_{t\to+\iy}W(t)>0,$$
so there exist local minimum points $t_n$ of $W$ such that $t_n\uparrow+\iy$ and
$W(t_n)\to 0$. 

By \eqref{eq-W}, there exist $T_2>T_0$ large and $\e_0\in (0, \frac12\limsup\limits_{t\to+\iy}W)$ small such that 
\begin{equation}\label{eq3-27}W''(t)>0\quad \text{whenever} \quad t\geq T_2\;\text{and}\;W(t)\le 2\e_0.\end{equation}
Therefore, for $n$ large, there exists $T_2<t_n^*<t_n<\bar{t}_n$ such that
\begin{equation}\label{eq3-32}
\begin{cases}W(t_n^\ast)=W(\bar{t}_n)=\e_0,\\
W'(t)<0\;\text{ for }\;t\in[t_n^\ast, t_n), \\
 W'(t)>0\;\text{ for }\;t\in(t_n, \bar{t}_n].
\end{cases} 
\end{equation}
Similarly as the proof of Lemma \ref{lemma03-3}, we have $\bar{\omega}_1(\cdot+t_n)\to 0$ and  $\bar{\omega}_2(\cdot+t_n)\to 0$ uniformly in $C_{loc}^2(\mathbb R^2)$ as $n\to+\iy$. 
Since \eqref{eq-W} and \eqref{eq3-2} together imply
\begin{equation}
W''(t)+C(t)W(t)=0,
\end{equation}
with $C(t)$ being bounded for $t>T_0$ large,
we can apply the Harnack inequality to obtain that $W(\cdot+t_n)/W(t_n)$ is locally uniformly bounded, so up to a subsequence, 
$$\frac{\bar{\omega}_1(\cdot+t_n)}{W(t_n)}\to v_1\ge 0,\quad
\frac{\bar{\omega}_2(\cdot+t_n)}{W(t_n)}\to v_2\ge 0$$
uniformly in $C_{loc}^2(\mathbb R)$. 
Moreover, $v_i''-\dd^2 v_i=0$, so $$v_i(t)=a_ie^{\dd t}+b_i e^{-\dd t}\quad\text{with}\quad a_i, b_i\ge 0.$$ Since $(v_1+v_2)(0)=1$ and $(v_1+v_2)'(0)=0$, we have $a_1+a_2=b_1+b_2=1/2$. Noting from \eqref{eq3-15} and $D_\iy=0$ that
$$D(t_n)=(\bar{\omega}_1(t_n)^2+\bar{\omega}_2(t_n)^2)O(e^{-t_n})=o(W(t_n)^2),$$
we deduce from \eqref{eq3-12} that
$$(v_1')^2(0)+(v_2')^2(0)=\dd^2v_1^2(0)+\dd^2v_2^2(0),$$
namely $$(a_1-b_1)^2+(a_2-b_2)^2=(a_1+b_1)^2+(a_2+b_2)^2,$$ which implies $a_1b_1+a_2b_2=0$.
Thus, $(a_1, a_2, b_1, b_2)=(0, \frac12, \frac12, 0)$ or $(\frac12, 0, 0, \frac12)$, that is,
\begin{equation}(v_1(t), v_2(t))=\left(\frac{1}{2}e^{-\dd t}, \frac{1}{2}e^{\dd t}\right)\quad\text{or}\quad \left(\frac{1}{2}e^{\dd t}, \frac{1}{2}e^{-\dd t}\right).\end{equation} In particular, $v_1(0)=v_2(0)=\frac{1}{2}$, so
$\frac{\bar{\omega}_1(t_n)}{W(t_n)}\to \frac{1}{2}$
and then \be\label{eq3-31}W(t_n)\le 3\bar{\omega}_1(t_n)\le Ce^{-\dd t_n}\quad\text{for $n$ large}.\ee 

Now we follow the argument of \cite[p.236]{CL3} to give a lower bound for $t_n-t_n^*$, from which and \eqref{eq3-31} we can  obtain a contradiction.
By \eqref{eq-W} and \eqref{eq3-32}, we have
$$\frac{d}{dt}\big[(W')^2-\delta^2W^2\big]\geq0\quad\text{for $t\in [t_n^*, t_n]$},$$
and so
$$W'(t)^2\leq \delta^2(W(t)^2-W(t_n)^2)\quad\text{for $t\in [t_n^*, t_n]$},$$
namely 
$-W'(t)\leq \delta\sqrt{W(t)^2-W(t_n)^2}$, or equivalently
$$dt\geq-\frac{dW}{\delta\sqrt{W(t)^2-W(t_n)^2}}\quad\text{for $t\in [t_n^*, t_n]$}.$$
Consequently, integrating over $[t_n^*, t_n]$ and using \eqref{eq3-31}, we obtain
{\allowdisplaybreaks
\begin{align}
t_n-t_n^\ast&\geq \int_{W(t_n)}^{W(t_n^*)}\frac{dW}{\delta\sqrt{W^2-W(t_n)^2}}
=\frac{1}{\delta}\int_1^{\frac{\e_0}{W(t_n)}}\frac{d\eta}{\sqrt{\eta^2-1}}\\
&=\frac{1}{\delta}\ln\left(\frac{\e_0}{W(t_n)}+\sqrt{\frac{\e_0^2}{W(t_n)^2}-1}\right)\ge\frac{1}{\delta}\ln\frac{\e_0}{W(t_n)}\nonumber\\
&\geq\frac{1}{\delta}\ln\e_0+t_n-C,\nonumber
\end{align}
}%
that is, $t_n^\ast\le C$ for any $n$ large, a contradiction with $t_n^\ast\to +\infty$.

Therefore, $\lim_{t\to+\iy}\bar{\omega}_2(t)=0$ and so
$\lim_{t\to+\iy}W(t)=0$. Applying Lemma \ref{lemma3-1}, we obtain that $v$ is also bounded for $|x|>0$ small.
 The proof is complete.
\end{proof}

For the case $D_\iy<0$, we need to discuss $N\geq 5$ and $N=4$ separately. The remaining case $N=3$ remains open.

\begin{lemma}\label{lemma2-9}
For $N\ge 5$ and $D_\iy<0$, there are constants $C_1, C_2>0$ such that for $|x|>0$ small,
$$C_1|x|^{-\dd}\le u(x), v(x)\le C_2|x|^{-\dd},$$
namely $(u,v)$ is both-singular.
\end{lemma}

\begin{proof}
By Lemma \ref{lemma2-6} and Theorem \ref{thm-as}, we only need to prove $\bar u(x), \bar v(x)\geq C|x|^{-\delta}$ for $|x|>0$ small, or equivalently, to prove
$$\liminf_{t\to+\iy}\bar{\omega}_1(t)>0,\quad\liminf_{t\to+\iy}\bar{\omega}_2(t)>0.$$
Assume by contradiction that $\liminf_{t\to+\iy}\bar{\omega}_1(t)=0$. If $\limsup_{t\to+\iy}\bar{\omega}_1(t)>0$, then there exists a sequence of local minimum points $t_n\uparrow+\iy$ of $\bar{\omega}_1$ such that $\bar{\omega}_1(t_n)\to 0$. Then by $\bar{\omega}_1''(t_n)\ge 0$ and \eqref{eq3-17}, we get 
$$\dd^2\bar{\omega}_1(t_n)\geq\bb (1+O(e^{-t_n})) \bar{\omega}_1(t_n)^{\frac{2^\ast}{2}-1}\bar{\omega}_2(t_n)^{\frac{2^\ast}{2}},
$$
which together with $N\geq 5$ implies $\bar\omega_2(t_n)\leq C\bar\omega_1(t_n)^{\frac{N-4}{N}}\to 0$, namely $\bar\omega_1(t_n)+\bar\omega_2(t_n)\to 0$, a contradiction with \eqref{eq3-22}.
This proves $$\lim_{t\to+\iy}\bar{\omega}_1(t)=0,\quad\text{and so}\quad\liminf_{t\to+\iy}\bar{\omega}_2(t)>0.$$
Then
{\allowdisplaybreaks
\begin{align*}
\bar{\omega}_1''&=\dd^2\bar{\omega}_1-\left(\mu_1\bar{\omega}_1^{2^\ast-1}+\bb \bar{\omega}_1^{\frac{2^\ast}{2}-1}\bar{\omega}_2^{\frac{2^\ast}{2}}\right)
(1+O(e^{-t}))\\
&\le \bar{\omega}_1^{\frac{2^\ast}{2}-1}\left(\dd^2\bar{\omega}_1^{2-\frac{2^\ast}{2}}
-\frac{\bb}{2}\bar{\omega}_2^{\frac{2^\ast}{2}}\right)<0,\quad\text{for $t$ large},
\end{align*}
}%
so $\bar{\omega}_1'(t)\downarrow 0$ as $t\uparrow +\iy$, namely $\bar{\omega}_1'(t)>0$ for $t$ large, a contradiction with $\bar{\omega}_1(t)\to 0$ as $t\to+\iy$. The proof is complete.
\end{proof}

\begin{lemma}\label{lemma2-10}
For $N=4$ and $D_\infty<0$, there exists $\alpha\in (0,1]$ and constants $C_1,C_2>0$ 
such that
$|x|>0$ small,
\begin{equation}\label{eq-alpha}
        C_1|x|^{-\alpha}\le u(x),v(x)\le C_2|x|^{-1},
\end{equation}
namely $(u, v)$ is both-singular.
Moreover, if we further assume
$$
        \beta>\max\{\mu_1,\mu_2\}\quad\text{or }\beta=\mu_1=\mu_2,
$$
then we can take $\alpha=1$ in \eqref{eq-alpha}, namely for $|x|>0$ small,
\begin{equation}\label{eq-alpha1}
        C_1|x|^{-1}\le u(x),v(x)\le C_2|x|^{-1}.
\ee
\end{lemma}

\begin{proof} Remark that the proof of Lemma \ref{lemma2-9} can not work for $N=4$, and here we need to develop different ideas.
First we prove \eqref{eq-alpha}. By Lemma \ref{lemma2-6} and Theorem \ref{thm-as}, we only need to prove the existence of $\alpha>0$ and $C_1>0$ such that $\bar u(x), \bar v(x)\geq C_1|x|^{-\alpha}$ for $|x|>0$ small.

Recall \eqref{eq3-23} and $N=4$ that $\bar u(r)+\bar v(r)\geq Cr^{-1}$ for $r>0$ small, then
\begin{align*}-(r^{3}\bar{u}'(r))'&=r^{3}(\mu_1\bar{u}^3+\bb \bar{u}\bar{v}^{2})(1+O(r))\\
&\geq Cr^{3}\bar{u}(\bar u+\bar{v})^{2}\geq Cr\bar u(r),\quad\text{for $r>0$ small,}\end{align*}
Since Lemma \ref{lemma2-2} implies $r^{3}\bar{u}'(r)\to 0$ as $r\to 0$, and \eqref{eq02-8} says that $\bar{u}(r)$ is decreasing with respect to $r>0$,
we get that  for $r>0$ small,
{\allowdisplaybreaks
\begin{align*}
-r^{3}\bar u'(r)\geq C\int_{0}^r \rho\bar{u}(\rho)d\rho
\ge C\bar{u}(r)\int_{0}^r \rho d\rho=:\alpha r^2\bar{u}(r),
\end{align*}
}%
where $\alpha>0$.
This implies $\frac{d}{dr}(r^{\alpha}\bar u(r))\leq 0$ and so $\bar u(r)\geq Cr^{-\alpha}$ for $r>0$ small. Similarly we can prove $\bar v(r)\geq Cr^{-\alpha}$ for $r>0$ small. This proves \eqref{eq-alpha}.

For the special case $\beta=\mu_1=\mu_2$, it follows from \cite[Theorem 1.3]{COS} that we can take $\alpha=1$.
Now we assume $\beta>\max\{\mu_1, \mu_2\}$ and prove \eqref{eq-alpha1}. By Lemma \ref{lemma2-6}, we only need to prove
\begin{equation}\label{eq3-40}\liminf_{t\to+\iy}\bar{\omega}_1(t)>0,\quad\liminf_{t\to+\iy}\bar{\omega}_2(t)>0.\end{equation}

Suppose by contradiction, without loss of generality, that
$$
\liminf_{t\to+\infty}\bar\omega_1(t)=0.
$$
By \eqref{eq3-22}, we have
$$
\liminf_{t\to+\infty}
\bigl(\bar\omega_1(t)+\bar\omega_2(t)\bigr)>0.
$$
Hence, there exist $c_0>0$ and $T_1>T_0$ such that
$\bar\omega_1(t)+\bar\omega_2(t)\ge c_0$ for $t\geq T_1$.
Set
$$
q(t):=\frac{\bar\omega_1(t)}{\bar\omega_2(t)}>0.
$$
Then
\be\label{eq3-37}
\liminf_{t\to+\infty} q(t)=0.
\ee

For $N=4$, we see from \eqref{eq3-17} that for $t$ large,
\begin{equation}\label{eq3-36}
\begin{cases}
\bar\omega_1''-\bar\omega_1+
(1+\varepsilon_1(t))\left(\mu_1\bar\omega_1^3
+\beta\bar\omega_1\bar\omega_2^2\right)=0,\\
\bar\omega_2''-\bar\omega_2+
(1+\varepsilon_2(t))\left(\mu_2\bar\omega_2^3
+\beta\bar\omega_2\bar\omega_1^2\right)=0,
\end{cases}
\end{equation}
where $\varepsilon_i(t)=O(e^{-t})$.
Using \eqref{eq3-36} and $\bar\omega_1''-q\bar\omega_2''
=\bar\omega_2 q''+2\bar\omega_2'q'$, we easily obtain
$$
q''+2\frac{\bar\omega_2'}{\bar\omega_2}q'
+\bar\omega_2^2q
\left[
(1+\varepsilon_1)\beta-(1+\varepsilon_2)\mu_2
+
\bigl((1+\varepsilon_1)\mu_1-(1+\varepsilon_2)\beta\bigr)q^2
\right]=0.
$$
Equivalently,
$$
\left(\bar\omega_2^2q'\right)'+\bar\omega_2^4q
\left[(1+\varepsilon_1)\beta-(1+\varepsilon_2)\mu_2+
\bigl((1+\varepsilon_1)\mu_1-(1+\varepsilon_2)\beta\bigr)q^2
\right]=0.
$$

Since $\beta>\mu_2$ and $\varepsilon_i(t)\to0$, we may choose $\eta>0$ small and
$T_2\ge T_1$ large such that, whenever $t\ge T_2$ and $0<q(t)\le\eta$,
$$
\left[(1+\varepsilon_1)\beta-(1+\varepsilon_2)\mu_2+
\bigl((1+\varepsilon_1)\mu_1-(1+\varepsilon_2)\beta\bigr)q(t)^2
\right]\ge c_1>0.
$$
Moreover, on the same set, it follows from
$\bar\omega_1(t)+\bar\omega_2(t)=\bar\omega_2(t)(q(t)+1)\geq c_0$ that
$0<c_2\le \bar\omega_2(t)\le C_2$.
Therefore, on every connected component of the set
$$\mathcal{T} :=\big\{t> T_2\;:\;q(t)<\eta\big\},$$
we have
$$
\left(\bar\omega_2^2q'\right)'\le -c_3q<0.
$$

Note from \eqref{eq3-37} that $\mathcal{T}\cap (R,+\infty)\neq \emptyset$ for all $R>T_2$. We claim that the open set $\mathcal T$ can not contain any other bounded connected component except at most the one with $T_2$ being its left endpoint. Indeed, suppose $(a,b)\subset\mathcal T$ is a bounded component with $b>a>T_2$, then $q(a)=q(b)=\eta$, so $q$ attains a local minimum at some point $s\in (a,b)$. Moreover, 
$q'(s)=0$, $q''(s)\ge0$, so
$$
0\leq\left(\bar\omega_2^2q'\right)'(s)\le -c_3q(s)<0,
$$
a contradiction.
Thus, $\mathcal T$ must contain an unbounded component $(a,+\infty)$.
Restricting to this component, we set
$$
        A(t):=\bar\omega_2(t)^2q'(t).
$$
Then $A'(t)\leq -c_3q(t)<0$. If $A(t_0)<0$ for some $t_0>a$, then for all $t>t_0$,
$$
        q'(t)=\frac{A(t)}{\bar\omega_2(t)^2}
        \le C_2^{-2} A(t_0)<0,
$$
so $q$ becomes negative in finite time, a contradiction. Hence $A(t)\ge0$ on
$(a,+\infty)$, i.e., $q'(t)\geq 0$ and so $q$ is non-decreasing on $(a, +\infty)$. But this contradicts with \eqref{eq3-37}. Therefore, \eqref{eq3-40} holds and so does \eqref{eq-alpha1}. The proof is complete.
\end{proof}

\begin{proof}[Proof of Theorems \ref{th2} and \ref{th20}]
Clearly Theorems \ref{th2} and \ref{th20} follow directly from Lemmas \ref{lemma2-8}, \ref{lemma2-9} and \ref{lemma2-10}.
\end{proof}

\section{Proof of Theorems \ref{thm:N3-semi} and \ref{th120}}

This section is devoted to the proofs of Theorems \ref{thm:N3-semi} and \ref{th120} by using  the stable invariant manifold theory of ODEs.

\begin{proof}[Proof of Theorem \ref{th120}]  Assume $N=4$ and
$$0<\beta<\mu_1,\quad \mu_2>0.$$
Then the system \eqref{eq3-9} becomes
\be\label{e0q3-9}
\begin{cases}
\begin{split}v_{1}''-v_{1}+\mu_1v_1^{3}+\bb v_1v_2^{2}=0\\
v_{2}''-v_{2}+\mu_2v_2^{3}+\bb v_2v_1^{2}=0\\
\end{split}  \quad   t\in\R.\end{cases}\ee
Define
$$\vec{y}(t)=\begin{pmatrix}y_1(t)\\
y_2(t)\\
y_3(t)\\
y_4(t)\end{pmatrix}:=\begin{pmatrix}v_1(t)\\
v_1'(t)\\
v_2(t)\\
v_2'(t)\end{pmatrix}$$
and
$$F(\vec{y}):=\begin{pmatrix}y_2\\y_1-\mu_1y_1^3-\beta y_1y_3^2\\y_4\\y_3-\mu_2y_3^3-\beta y_3y_1^2\end{pmatrix},$$
then \eqref{e0q3-9} is equivalent to the system $\vec{y}'(t)=F(\vec{y}(t))$.
Given any $\vec{a}\in \mathbb{R}^4$, we consider the initial value problem
\begin{equation}
\label{eqqq0-1}
\begin{cases}\vec{y}'(t)=F(\vec{y}(t)),\\
\vec{y}(0)=\vec{a}.
\end{cases}
\end{equation}
By the theory of ODEs, we know that there is a unique solution $\vec{y}(t; \vec{a})$. Since
\begin{align*}
&\frac{1}{2}(y_2^2+y_4^2-y_1^2-y_3^2)(t)+\frac{1}{4}(\mu_1y_1^{4}+2\bb y_1^{2}y_3^{2}+\mu_2y_3^{4})(t)\\
\equiv& \frac{1}{2}(a_2^2+a_4^2-a_1^2-a_3^2)+\frac{1}{4}(\mu_1a_1^{4}+2\bb a_1^{2}a_3^{2}+\mu_2a_3^{4}),
\end{align*}
so the maximum interval for $t$ is $\mathbb{R}$, namely $\vec{y}(t; \vec{a})$ is well-defined for $t\in\mathbb R$. 

Denote
$$\vec{a}_0:=\begin{pmatrix}\frac{1}{\sqrt{\mu_1}}\\
0\\
0\\
0\end{pmatrix},$$
then $F(\vec{a}_0)=0$ and so $\vec{y}(t;\vec{a}_0)\equiv \vec{a}_0$. Such $\vec{a}_0$ is called an equilibrium point of the system \eqref{eqqq0-1} in the theory of ODEs (cf. \cite{P}).
A direct computation gives
$$DF(\vec{a}_0)=\begin{pmatrix}0&1&0&0\\-2&0&0&0\\0&0&0&1\\0&0&1-\frac{\beta}{\mu_1}&0\end{pmatrix}.$$
Thanks to our assumption $\beta<\mu_1$,
we see that the eigenvalues of $DF(\vec{a}_0)$ are $\pm\sqrt{2} i$ and $\pm\sqrt{1-\frac{\beta}{\mu_1}}$. Denote by
$$\lambda_0=-\sqrt{1-\frac{\beta}{\mu_1}}<0$$
to be the unique eigenvalue with negative real part
and the eigenspace of $\lambda_0$ by
$$E_0:=\mathbb{R}\begin{pmatrix}0\\
0\\
1\\
\lambda_0\end{pmatrix}.$$
Then by the stable manifold theorem (see e.g. \cite[Theorem in p.116]{P} or \cite[Theorem 3.2.1]{GH}), there exists a one-dimensional smooth and stable manifold $E^{s}\subset\mathbb R^4$ satisfying the following properties:
\begin{itemize}
\item[(i)] $\vec{a}_0\in E^{s}$;
\item[(ii)] $E^s$ is tangent to the eigenspace $E_0$ at $\vec{a}_0$;
\item[(iii)] $E^s$ is positively invariant under the flow $\vec{y}(t,\cdot)$: $\vec{y}(t; E^s)\subset E^s$ for any $t\geq 0$;
\item[(iv)] For any $\vec{a}\in E^s$, 
\begin{equation}\label{eeff}\lim_{t\to+\infty}\vec{y}(t; \vec{a})=\vec{a}_0.\end{equation}
\end{itemize}

By Properties (i)-(ii), we can take $\vec{a}\in E^s$ close to $\vec{a}_0$ such that $a_1>0$ and $a_3\neq 0$. Then the solution $\vec{y}(t):=\vec{y}(t; \vec{a})$ of \eqref{eqqq0-1} satisfies
$$y_1(t)\not\equiv 0,\quad y_3(t)\not\equiv 0.$$
By $\beta<\mu_1$, we can take $\varepsilon\in (0, \frac{1}{\sqrt{\mu_1}})$ small such that
\begin{equation}\label{qeeq-2}
1-\mu_2\varepsilon^2-\beta\Big(\frac{1}{\sqrt{\mu_1}}+\varepsilon\Big)^2>0.
\end{equation}
Then by \eqref{eeff}, there exists $T>0$ large such that $y_3(T)\neq 0$ and
\begin{equation}\label{qeeq-3}\Big|y_1(t)-\frac{1}{\sqrt{\mu_1}}\Big|+|y_2(t)|+|y_3(t)|+|y_4(t)|<\varepsilon,\quad\forall t\geq T.\end{equation}
In particular, $y_1(t)>0$ for all $t\geq T$. Without loss of generality, we may assume $y_3(T)>0$ (otherwise, we just need to replace $(y_3, y_4)$ with $(-y_3,-y_4)$ because $(y_1, y_2, -y_3,-y_4)$ is also a solution of \eqref{eqqq0-1}). We claim that
$$y_3(t)>0,\quad\forall t\geq T.$$
Assume by contradiction that $y_3(t)$ changes sign in $[T, +\infty)$. Since $y_3(t)\to 0$ as $t\to+\infty$, then there exists $t_1>T$ being a local minimum point of $y_3(t)$ with $y_3(t_1)<0$, so $y_3''(t_1)\geq 0$. However, \eqref{qeeq-2}-\eqref{qeeq-3} imply
$$y_3''(t_1)=y_4'(t_1)=y_3(t_1)(1-\mu_2y_3(t_1)^2-\beta y_1(t_1)^2)<0,$$
a contradiction. Thus $y_3(t)\geq 0$ for any $t\geq T$, and then the strong maximum principle implies $y_3(t)>0$ for any $t\geq T$.

Let $$v_1(t):=y_1(t+T+\log 2),\quad v_2(t):=y_3(t+T+\log 2).$$Then $(v_1, v_2)$ is a solution of \eqref{e0q3-9} and satisfies
$$v_1(t), v_2(t)>0\quad\text{for }t\geq -\log 2,\quad \lim_{t\to+\infty}v_1(t)=\frac{1}{\sqrt{\mu_1}},\quad \lim_{t\to+\infty}v_2(t)=0.$$
In particular, the corresponding $K$ defined in \eqref{eq003-2} is given by $K=\frac{-1}{4\mu_1}<0$.
Define
$$u(x):=|x|^{-1}v_1(-\log |x|),\quad  v(x):=|x|^{-1}v_2(-\log |x|),\quad |x|\leq 2,$$
Then $(u,v)$ is a positive and radially symmetric solution of \eqref{eq0} with $D(u,v)<0$ and
$$\lim_{x\to 0}|x|u(x)=\frac{1}{\sqrt{\mu_1}},\quad \lim_{x\to 0}|x|v(x)=0.$$
The proof is complete.
\end{proof}

\begin{proof}[Proof of Theorem \ref{thm:N3-semi}] Assume $N=3$ and $\mu_1, \mu_2, \beta>0$.
Then the system \eqref{eq3-9} becomes
\begin{equation}\label{eq:N3-ODE}
\begin{cases}
v_1''-\dfrac14v_1+\mu_1v_1^5+\beta v_1^2v_2^3=0,\\[2mm]
v_2''-\dfrac14v_2+\mu_2v_2^5+\beta v_2^2v_1^3=0.
\end{cases}
\end{equation}
Set
\begin{equation}\label{eq:a-star}
a_*:=(4\mu_2)^{-1/4}.
\end{equation}
Then $(v_1,v_2)=(0,a_*)$ is a solution of \eqref{eq:N3-ODE}.
As in the proof of Theorem  \ref{th120}, we define
\[
\vec y(t):=
\begin{pmatrix}
y_1(t)\\y_2(t)\\y_3(t)\\y_4(t)
\end{pmatrix}
:=
\begin{pmatrix}
v_1(t)\\v_1'(t)\\v_2(t)\\v_2'(t)
\end{pmatrix}
\]
and
\begin{equation}\label{eq:F-vectorfield}
F(\vec y):=
\begin{pmatrix}
y_2\\[1mm]
\dfrac14y_1-\mu_1y_1^5-\beta y_1^2y_3^3\\[1mm]
y_4\\[1mm]
\dfrac14y_3-\mu_2y_3^5-\beta y_3^2y_1^3
\end{pmatrix}.
\end{equation}
Then \eqref{eq:N3-ODE} is equivalent to
\begin{equation}\label{eq:first-order}
\vec y'(t)=F(\vec y(t)).
\end{equation}
Let
\begin{equation}\label{eq:a0}
\vec a_0:=
\begin{pmatrix}
0\\0\\a_*\\0
\end{pmatrix}.
\end{equation}
Clearly $F(\vec a_0)=0$ and so $\vec{y}(t;\vec{a}_0)\equiv \vec{a}_0$.
A direct computation gives
\begin{equation}\label{eq:linearization}
DF(\vec a_0)=
\begin{pmatrix}
0&1&0&0\\
\frac14&0&0&0\\
0&0&0&1\\
0&0&-1&0
\end{pmatrix},
\end{equation}
hence the eigenvalues of $DF(\vec a_0)$ are
$
\pm\frac12$ and $\pm i$.
In particular, the unique eigenvalue with negative real part is
\[
\lambda_0=-\frac12=-\delta,
\]
and its eigenspace is
\begin{equation}\label{eq:E0}
E_0
=
\R
\begin{pmatrix}
1\\-\frac12\\0\\0
\end{pmatrix}.
\end{equation}
Again by the stable manifold theorem, there exists a one-dimensional smooth and stable manifold $E^{s}\subset\mathbb R^4$ satisfying the following properties:
\begin{itemize}
\item[(i)] $\vec{a}_0\in E^{s}$;
\item[(ii)] $E^s$ is tangent to the eigenspace $E_0$ at $\vec{a}_0$;
\item[(iii)] $E^s$ is positively invariant under the flow $\vec{y}(t,\cdot)$: $\vec{y}(t; E^s)\subset E^s$ for any $t\geq 0$;
\item[(iv)] For any $\vec{a}\in E^s$, 
\begin{equation}\label{eeff-N3}\lim_{t\to+\infty}\vec{y}(t; \vec{a})=\vec{a}_0.\end{equation}
\end{itemize}

Compared with the four-dimensional construction in Theorem \ref{th120}, here we need to
use the precise decay rate along $E^s$ for the three-dimensional case.
Since the first component of the stable eigenvector in \eqref{eq:E0} is nonzero, after shrinking $E^s$ if necessary we may use $y_1$ as a local coordinate on $E^s$.  Thus there exist smooth functions $\psi_j$ such that, for $|s|$ small,
\begin{equation}\label{eq:stable-graph}
E^s=
\left\{
\begin{pmatrix}
s\\ \psi_2(s)\\ a_*+\psi_3(s)\\ \psi_4(s)
\end{pmatrix}: |s|<\rho
\right\},
\end{equation}
with
\begin{equation}\label{eq:stable-tangent}
\psi_2(s)=-\frac12s+O(s^2),
\quad
\psi_3(s)=O(s^2),
\quad
\psi_4(s)=O(s^2)
\quad\text{as }s\to0.
\end{equation}

Take a point $\vec a\in E^s$ on the branch $s>0$ of $E^s$, sufficiently close to $\vec a_0$, and let $\vec y(t)=\vec y(t; \vec a)$ be the corresponding solution of \eqref{eq:first-order} with the initial value $\vec y(0)=\vec a$. In particular, $y_1(t)>0$ for $t\geq 0$ small. 
We claim that
\begin{equation}\label{eq:y1-positive}
y_1(t)>0\qquad\text{for all }\;t\ge0.
\end{equation}
Indeed, after shrinking the local manifold $E^s$, we may assume $E^s\cap\{\vec y\in \mathbb R^4 : y_1=0\}=\{\vec a_0\}$.  Hence, if $y_1(t_0)=0$ for some $t_0>0$, it follows from Property (iii) that $y(t_0)=\vec{a}_0$, so $y(t)\equiv \vec a_0$ by the uniqueness of solutions, a contradiction to our choice of the initial value. This proves \eqref{eq:y1-positive}.
Similarly, by choosing the neighborhood of $\vec a_0$ sufficiently small, we also have
\begin{equation}\label{eq:y3-positive}
y_3(t)\geq\frac{a_*}{2}>0\qquad\text{for all }t\ge0.
\end{equation}

By \eqref{eq:stable-graph}--\eqref{eq:stable-tangent},
\begin{equation}\label{eq:s-reduced}
y_1'(t)=y_2(t)=\psi_2(y_1(t))
=-\frac12y_1(t)+O(y_1(t)^2),\quad t\geq 0.
\end{equation}
Note from \eqref{eeff-N3} that $y_1(t)\to0$ as $t\to+\infty$. From here and \eqref{eq:s-reduced}, there is $t_1>0$ such that
\[
y_1'(t)\le-\frac14y_1(t),\quad t\geq t_1,
\]
so $y_1(t)=O(e^{-t/4})$ for $t\geq t_1$ and then $y_1\in L^1([0,+\infty))$.  Dividing \eqref{eq:s-reduced} by $y_1(t)>0$, we obtain
\[
\frac{d}{dt}\log\bigl(e^{t/2}y_1(t)\bigr)=O(y_1(t)).
\]
Since $y_1\in L^1([0,+\infty))$, the right-hand side is integrable.  Therefore,
\begin{equation}\label{eq:A0}
A_0:=\lim_{t\to+\infty}e^{t/2}y_1(t)
\end{equation}
exists and satisfies $A_0\in(0,+\infty)$.  In particular,
\begin{equation}\label{eq:y1-asym}
y_1(t)=A_0e^{-t/2}(1+o(1))\quad\text{as }t\to+\infty.
\end{equation}
Furthermore, by \eqref{eq:stable-tangent},
\begin{equation}\label{eq:y3-asym}
y_3(t)=a_*+O(y_1(t)^2)=a_*+O(e^{-t})\quad\text{as }t\to+\infty.
\end{equation}

Define
\[
v_1(t):=y_1(t+\log2),
\quad
v_2(t):=y_3(t+\log2),
\quad t\ge-\log2.
\]
Then $(v_1,v_2)$ is a positive solution of \eqref{eq:N3-ODE} on $[-\log2,+\infty)$ and
\begin{equation}\label{eq:v-asym}
v_1(t)=Ae^{-t/2}(1+o(1)),
\quad
v_2(t)=a_*+O(e^{-t}),\quad\text{as }t\to+\infty,
\end{equation}
for some $A>0$.
Define
\[
u(x):=|x|^{-1/2}v_1(-\log|x|),
\quad
v(x):=|x|^{-1/2}v_2(-\log|x|),
\quad |x|\leq 2.
\]
Then $(u,v)$ is a positive radially symmetric solution of \eqref{eq0}.  From \eqref{eq:v-asym},
\[
\lim_{|x|\to0}u(x)=A>0,
\quad
\lim_{|x|\to0}|x|^{1/2}v(x)=a_*=(4\mu_2)^{-1/4}.
\]
Thus $u$ is bounded near the origin whereas $v(x)\to+\infty$ as $|x|\to0$, namely $(u,v)$ is semi-singular.

It remains to compute the Pohozaev invariant.  Since $(v_1,v_1',v_2,v_2')\to(0,0,a_*,0)$ as $t\to+\infty$, it follows from \eqref{eq003-2}, $2^*=6$, $\delta=1/2$ and $\mu_2a_*^4=1/4$ that
\begin{align*}
K
=-\frac18a_*^2+\frac16\mu_2a_*^6
=-\frac1{24\sqrt{\mu_2}}<0.
\end{align*}
Consequently,
\[
D(u,v)=|\mathbb S^2|K=-\frac{\pi}{6\sqrt{\mu_2}}<0.
\]

Finally, 
the last assertion follows by repeating the same construction around the solution 
$((4\mu_1)^{-1/4},0)
$
of \eqref{eq:N3-ODE}. This completes the proof.
\end{proof}

\section{Removable Singularity when $D_\infty=0$}
\label{sec-5}

Recall that if the isolated singularity $0$ is removable for both $u$ and $v$, i.e., $u, v$ are both bounded for $|x|>0$ small, then $D_\infty=0$.
The purpose of this section is to prove the converse statement Theorem \ref{th3}.

\begin{theorem}[=Theorem \ref{th3}]\label{thm3-1}
Assume that $D_\infty=0$. If either
$N\ge 5$, or
$$
N=4,\qquad\beta>\min\{\mu_1,\mu_2\}\ \text{or}\ \beta=\mu_1=\mu_2,
$$
then the isolated singularity $0$ is removable for both $u$ and $v$.
\end{theorem}

Remark that those methods in \cite{CGS, ChenLin1} can not be applied in the proof of Theorem \ref{thm3-1}. Here we develop different ideas to treat $N\geq 5$ and $N=4$ separately. The proof of Theorem \ref{thm3-1} is quite long, and 
we divide the proof into several steps. 
Throughout this section, we always let $D_\infty=0$. Then by Lemma \ref{lemma3-1}, to prove Theorem \ref{thm3-1} is equivalent to prove
\begin{equation}\label{eq4-1}
W(t):=\bar\omega_1(t)+\bar\omega_2(t)\to 0\quad\text{as }t\to+\infty.
\end{equation}
Recalling \eqref{eq3-22-1} that there holds
\begin{equation}\label{eq4-2}
\liminf_{t\to+\infty}W(t)=0.
\ee
The following lemma studies the structure of valley points if \eqref{eq4-1} does not hold. 

\begin{lemma}\label{lemma3-2}
Suppose \eqref{eq4-1} does not hold. Then there exist local minimum points
$t_k\to+\infty$ of $W(t)$ such that
\begin{equation}\label{eq4-4}
m_k:=W(t_k)\to0,\qquad W'(t_k)=0.
\end{equation}
For every such sequence, after passing to a subsequence, the normalized functions
$$
Z_{i,k}(s):=\frac{\bar\omega_i(t_k+s)}{m_k},\quad i=1,2,
$$
converge uniformly in $C^2_{\mathrm{loc}}(\mathbb R)$ to one of the following two profiles:
\be\label{eq4-3}
(Z_1(s),Z_2(s))=
\left(\frac12e^{\delta s},\frac12e^{-\delta s}\right)
\quad\text{or}\quad
(Z_1(s),Z_2(s))=
\left(\frac12e^{-\delta s},\frac12e^{\delta s}\right).
\ee
\end{lemma}

\begin{proof}
The proof is similar to that of Lemma \ref{lemma2-8}.
Since we assume $W(t)\not\to0$ as $t\to+\infty$, it follows from \eqref{eq4-2} that there are local minimum points
$t_k\to+\infty$ of $W(t)$ such that
$$
        m_k:=W(t_k)\to0,\qquad W'(t_k)=0.
$$
Now we classify the profile near $t_k$. Recalling \eqref{eq-W} that $W''-\delta^2W\le0$ for $t$ large, the comparison formula yields
\be\label{eq4-60}
        W(t_k+s)\le m_k\cosh(\delta s)=m_k\frac{e^{\delta s}+e^{-\delta s}}{2},\quad\text{for every fixed $s$.}
\ee
Hence the normalized functions
$Z_{i,k}(s)=\frac{\bar\omega_i(t_k+s)}{m_k}$
are locally uniformly bounded. Then up to a subsequence, $Z_{i,k}\to Z_i\geq 0$ uniformly in $C_{loc}^2(\mathbb R)$, where $Z_i$ satisfies $ Z_i''-\delta^2Z_i=0$ and so
$$
        Z_i(s)=a_ie^{\delta s}+b_ie^{-\delta s},
        \qquad a_i,b_i\ge0.
$$
From
$$
        Z_1(0)+Z_2(0)=1,\qquad
        Z_1'(0)+Z_2'(0)=0,
$$
we obtain
$a_1+a_2=b_1+b_2=\frac12.
$
 Again $D_\infty=0$ and \eqref{eq3-15} imply $\frac{D(t_k)}{m_k^2}=O(e^{-t_k})\to 0$,  then by \eqref{eq3-12}, we get
$$
        \sum_{i=1}^2
        \left(Z_i'(0)^2-\delta^2Z_i(0)^2\right)=0,
$$
so the same argument as Lemma \ref{lemma2-8} implies $(a_1,a_2,b_1,b_2)=(0,\frac12,\frac12,0)$ or $(\frac12,0,0,\frac12)$, namely \eqref{eq4-3} holds.
The proof is complete.
\end{proof}

We now set up the contradiction argument used in the sequel. Suppose that $W(t)\not\to0$ as $t\to+\infty$. Recalling \eqref{eq-W} that
\begin{equation}\label{eq-W1}C_1W^{2^\ast-1}\le -(W''-\dd^2W)\le C_2W^{2^\ast-1},\quad\text{for $t> T_0$ large}.\end{equation}
Fix a small number $\varepsilon>0$ such that 
\be\label{eq4-5}
C_2\varepsilon^{2^*-2}\le \frac{\delta^2}{2},\quad \varepsilon<\frac12\limsup_{t\to+\infty}W(t).
\ee
By Lemma~\ref{lemma3-2}, there
exist local minimum points $t_k\to+\infty$ with
$$
\varepsilon>m_k:=W(t_k)\to0,\qquad W'(t_k)=0,
$$
namely $t_k\in\{t:W(t)<\varepsilon\}$. It follows from \eqref{eq4-5} that for $k$ large, the corresponding left and right exit times 
$$
\tau_k^-:=\sup\{t<t_k: W(t)=\varepsilon\},\qquad
\tau_k^+:=\inf\{t>t_k: W(t)=\varepsilon\}
$$
are well-defined and satisfies
\begin{equation}\label{eq4-7}
W(\tau_k^-)=W(\tau_k^+)=\varepsilon,\ \
\text{and}\ \ W(t)<\varepsilon,\ \ t\in(\tau_k^-,\tau_k^+).
\end{equation}
On each such interval,
$$
        W''\ge \delta^2W-C_2W^{2^*-1}
        \ge \frac{\delta^2}{2}W>0,
$$
so
\be\label{eq4-8}
        W'(t)<0\quad \text{on }[\tau_k^-,t_k),
        \qquad
        W'(t)>0\quad \text{on }(t_k,\tau_k^+].
\ee
Thus, $t_k$ is the unique local minimum point of $W$ on the component $(\tau_k^-, \tau_k^+)$ of $\{W<\varepsilon\}$.
In particular, this implies
\be\label{eq4-060}
t_{k-1}<\tau_{k-1}^+<\tau_k^-<t_k\quad\text{for $k$ large.}
\ee

The following lemma shows that at every exit, one of the two exponents will be significantly larger than the other.

\begin{lemma}\label{lemma3-3}
Under the assumptions above, 
\be\label{eq4-13}
\lim_{k\to+\infty}\min\{\bar\omega_1(\tau_k^+),
\bar\omega_2(\tau_k^+)\}=0,
\ee
\be\label{eq4-14}
\lim_{k\to+\infty}\min\{\bar\omega_1(\tau_k^-),
\bar\omega_2(\tau_k^-)\}=0.
\ee
\end{lemma}

\begin{proof}
We first consider the right exit $\tau_k^+$ and argue by contradiction. Recall $\bar\omega_1(\tau_k^+)+\bar\omega_2(\tau_k^+)=W(\tau_k^+)=\varepsilon$.  We assume the existence of $a_0\in (0, \varepsilon/2)$ such that up to a subsequence,
$$
\bar\omega_1(\tau_k^+)\ge a_0,\qquad \bar\omega_2(\tau_k^+)\ge a_0
$$
for all large $k$. Since $\bar\omega_1'$ and $\bar\omega_2'$ are bounded by \eqref{eq3-3-2},
there exist constants $\ell>0$ and $a_1>0$, independent of $k$, such that
\be\label{eq4-9}
\bar\omega_1(t)\ge a_1,\quad \bar\omega_2(t)\ge a_1,\qquad \forall\, t\in[\tau_k^+,\tau_k^++\ell].
\ee

We first estimate the distance from the valley $t_k$ to the right exit
$\tau_k^+$. Recall \eqref{eq-W1} that
$$
W''\ge \delta^2W-C_2W^{2^*-1}.
$$
Multiplying this inequality by $2W'>0$ and integrating from $t_k$ to $t\in(t_k,\tau_k^+]$, we
obtain
$$
W'(t)^2\ge\delta^2\bigl(W(t)^2-m_k^2\bigr)-\frac{2C_2}{2^\ast}\bigl(W(t)^{2^*}-m_k^{2^*}\bigr),\quad\text{for }t\in(t_k,\tau_k^+].
$$
Since $\varepsilon$ satisfies \eqref{eq4-5}, the mean value theorem gives 
$$
\frac{2C_2}{2^\ast}(y^{2^*}-m^{2^*})\le\frac{\delta^2}{2}(y^2-m^2),\quad\forall\,0<m<y\le\varepsilon.
$$
Thus, using $y=W(t)$ as a variable,
\begin{align}\label{eq4-10}
\tau_k^+-t_k=\int_{m_k}^{\varepsilon}\frac{dy}{W'}
\le&\int_{m_k}^{\varepsilon}
\frac{dy}{\sqrt{\delta^2(y^2-m_k^2)
-\frac{2C_2}{2^\ast}(y^{2^*}-m_k^{2^*})}}\nonumber\\
\le&\frac1\delta\int_{m_k}^{\varepsilon}
\frac{dy}{\sqrt{y^2-m_k^2}}
+C\int_{m_k}^{\varepsilon}
\frac{y^{2^*}-m_k^{2^*}}{(y^2-m_k^2)^{3/2}}\,dy,
\end{align}
where we have used the inequality 
$$
(A-B)^{-1/2}\le A^{-1/2}+\sqrt{2}\ A^{-3/2}B,\quad\text{for}\;0\leq B\leq A/2.
$$
Then the first integral 
$$
\frac1\delta\int_{m_k}^{\varepsilon}
\frac{dy}{\sqrt{y^2-m_k^2}}=
\frac1\delta\operatorname{arcosh}\frac{\varepsilon}{m_k}\le\frac1\delta\ln\frac{2\varepsilon}{m_k},
$$
and the second integral can be estimated as follows (Set
$y=m_k z$ and use $m_k<\varepsilon$):
{\allowdisplaybreaks
\begin{align*}
\int_{m_k}^{\varepsilon}
\frac{y^{2^*}-m_k^{2^*}}{(y^2-m_k^2)^{3/2}}\,dy
&=m_k^{2^*-2}\int_1^{\varepsilon/m_k}
\frac{z^{2^*}-1}{(z^2-1)^{3/2}}\,dz\\
&\leq m_k^{2^*-2}\left(\int_1^{2}
C(z-1)^{-\frac12}\,dz+\int_2^{\varepsilon/m_k}
\frac{z^{2^*}}{(z^2/2)^{3/2}}\,dz\right)\\
&\leq m_k^{2^*-2}\big(C+C(\varepsilon/m_k)^{2^\ast-2}\big)\leq C\varepsilon^{2^\ast-2}.
\end{align*}
}%
 Hence
$$
\tau_k^+-t_k\le
\frac1\delta\ln\frac{2\varepsilon}{m_k}+C_\varepsilon.
$$
Equivalently,
\be\label{eq4-11}
e^{-\delta(\tau_k^+-t_k)}\ge C_\varepsilon m_k.
\ee

We now use the valley profile from Lemma \ref{lemma3-2}. After passing to a further subsequence, assume first that
\be\label{eq4-12}
\frac{1}{m_k}(\bar\omega_1,\bar\omega_2)(t_k+s)
\longrightarrow
\left(\frac12e^{\delta s},\frac12e^{-\delta s}\right)
\qquad\text{in }C^2_{\mathrm{loc}}(\mathbb R).
\end{equation}

Write the second averaged equation in the form
$$
\bar\omega_2''-\delta^2\bar\omega_2=-G_2(t),
$$
where
$$
G_2(t):=(1+O(e^{-t}))
\left(\mu_2\bar\omega_2^{2^*-1}
+\beta \bar\omega_2^{\frac{2^*}{2}-1}\bar\omega_1^\frac{2^*}{2}\right)>0
$$
for large $t$. Recalling \eqref{eq02-9}, we define
$$
P_2(t):=\bar\omega_2'(t)+\delta \bar\omega_2(t)>0.
$$
Then \eqref{eq4-12} implies $P_2(t_k)=o(m_k)$. Clearly,
$$
\frac{d}{dt}\left(e^{-\delta t}P_2(t)\right)
=e^{-\delta t}(\bar\omega_2''-\delta^2\bar\omega_2)
=-e^{-\delta t}G_2(t).
$$
Integrating over $[t_k, \tau_k^++\ell]$, we obtain
\begin{align*}e^{-\delta t_k}P_2(t_k)&=e^{-\delta( \tau_k^++\ell)}P_2( \tau_k^++\ell)+\int_{t_k}^{ \tau_k^++\ell}e^{-\delta s}G_2(s)\,ds\\
&\geq \int_{t_k}^{ \tau_k^++\ell}e^{-\delta s}G_2(s)\,ds\geq \int_{\tau_k^+}^{ \tau_k^++\ell}e^{-\delta s}G_2(s)\,ds,\end{align*}
in particular,
$$
P_2(t_k)\ge\int_{\tau_k^+}^{\tau_k^++\ell}
e^{-\delta(s-t_k)}G_2(s)\,ds.
$$
On $[\tau_k^+,\tau_k^++\ell]$ we have $\bar\omega_1,\bar\omega_2\ge a_1$, and hence
$G_2(s)\ge c_1>0$.
Therefore, together with \eqref{eq4-11}, we obtain
$$
o(m_k)=P_2(t_k)\ge 
c_1\int_{\tau_k^+}^{\tau_k^++\ell} e^{-\delta(s-t_k)}\,ds
=c_2 e^{-\delta(\tau_k^+-t_k)}\ge c_3m_k,
$$
for all large $k$, which is a contradiction. 
For the other case
$$
\frac{1}{m_k}(\bar\omega_1,\bar\omega_2)(t_k+s)
\longrightarrow
\left(\frac12e^{-\delta s},\frac12e^{\delta s}\right)\qquad\text{in }C^2_{\mathrm{loc}}(\mathbb R),
$$
a similar argument gives
$o(m_k)=\bar\omega_1'(t_k)+\delta \bar\omega_1(t_k)\geq c_4m_k$ for large $k$, also a contradiction. This proves \eqref{eq4-13}.

Clearly, \eqref{eq4-14} can be proved by the same argument after the change of variable $s\mapsto -s$ around $t_k$, and we omit the details here. 
\end{proof}

Now we are ready to prove Theorem \ref{thm3-1}. We will treat $N\geq 5$ and $N=4$ separately.


\begin{proof}[Proof of Theorem \ref{thm3-1} for $N\geq 5$]
Assume by contradiction that $W(t)\not\to0$ as $t\to+\iy$. By \eqref{eq4-13}, we may assume, for instance, that after passing to a subsequence,
$\bar\omega_2(\tau_k^+)\to0$.
Then necessarily
$\bar\omega_1(\tau_k^+)\to\varepsilon$.
Define
$$
\psi_{i,k}(s):=\bar\omega_i(\tau_k^++s),\qquad i=1,2.
$$
Like before, after passing to a further subsequence,
$$
\psi_{i,k}\to \psi_i\geq 0
\qquad\text{in } C^2_{\mathrm{loc}}(\mathbb R),
$$
where $(\psi_1, \psi_2)$ solves the system \eqref{eq3-9} and satisfies
$\psi_2(0)=\psi_2'(0)=0$, $\psi_1(0)=\varepsilon$. Then the maximum principle or a standard ODE analysis implies $\psi_2\equiv 0$.

Since $\frac{2^*}{2}-1<1$ for $N\geq 5$, we may
choose $\rho>0$ such that
$$
\delta^2 y\le
\frac{\beta}{8}\left(\frac{\varepsilon}{4}\right)^\frac{2^*}{2} y^{\frac{2^*}{2}-1},
\qquad \forall\, y\in [0,\rho].
$$
Choose $\ell>0$ so small that $\psi_1(s)\ge \frac{\varepsilon}{2}$ for $|s|\le \ell$.


Then for $k$ large,
$\psi_{1,k}(s)\ge \frac{\varepsilon}{4}$ 
and $\psi_{2,k}(s)\leq \rho$ for $|s|\le \ell$, so the equation for
$\psi_{2,k}$ gives
$$
\psi_{2,k}''=
\delta^2\psi_{2,k}-(1+O(e^{-\tau_k^+-s}))
\left(\mu_2\psi_{2,k}^{2^*-1}
+\beta\psi_{2,k}^{\frac{2^*}{2}-1}\psi_{1,k}^\frac{2^*}{2}\right)
\le-c_2\psi_{2,k}^{\frac{2^*}{2}-1}
$$
on $[-\ell,\ell]$.

We now use the following elementary one-dimensional observation. If
$0<q<1$, $y(s)>0$ on $[-\ell,\ell]$, and
$y''(s)\le -c y(s)^q$, then
$$
\sup_{s\in[-\ell,\ell]}y(s)\ge c_0,
$$
where $c_0>0$ depends only on $c,q,\ell$. To prove this assertion, we take
$$
\varphi(s)=\cos\frac{\pi s}{2\ell}.
$$
Then $\varphi>0$ on $(-\ell,\ell)$, $\varphi(\pm\ell)=0$, and
$\varphi''=-\lambda\varphi$,
where $\lambda=\left(\frac{\pi}{2\ell}\right)^2$.
Multiplying the inequality $y''\le -cy^q$ by $\varphi$ and integrating by parts gives
$$
c\int_{-\ell}^{\ell}y^q\varphi
\le-\int_{-\ell}^{\ell}y''\varphi=\lambda\int_{-\ell}^{\ell}y\varphi-\frac{\pi}{2\ell}(y(\ell)+y(-\ell))\le\lambda\int_{-\ell}^{\ell}y\varphi.
$$
Denote $M=\sup_{[-\ell,\ell]}y$, then by $q-1<0$ we get
$$
cM^{q-1}\int_{-\ell}^{\ell} y\varphi
\leq c\int_{-\ell}^{\ell}y^q\varphi\le\lambda\int_{-\ell}^{\ell} y\varphi,
$$
and hence
$M\ge \left(\frac{c}{\lambda}\right)^{1/(1-q)}$.
This proves the observation.

Applying this observation to $y=\psi_{2,k}$ with $q=\frac{2^*}{2}-1$, we obtain
$$
\sup_{[-\ell,\ell]}\psi_{2,k}\ge c_0>0,\quad\text{for all $k$ large},
$$
which contradicts with the uniform convergence $\psi_{2,k}\to0$ on
$[-\ell,\ell]$.

Therefore, $W(t)\to 0$ as $t\to+\iy$. Applying Lemma \ref{lemma3-1}, we conclude that the isolated singularity $0$ is removable for both $u$ and $v$. This completes the proof of Theorem \ref{thm3-1} for $N\geq 5$.
\end{proof}

We now treat the remaining case $N=4$ under a strong coupling assumption.


\begin{proof}[Proof of Theorem \ref{thm3-1} for $N=4$]
For $N=4$, we further assume 
$$ \beta>\min\{\mu_1,\mu_2\}\quad \text{or}\quad \beta=\mu_1=\mu_2.$$
Clearly the proof for $N\geq 5$ does not work, and here we need to develop different ideas.
Again we argue by contradiction and assume that $W(t)\not\to0$ as $t\to+\iy$. By
Lemma~\ref{lemma3-2}, there exists a sequence of local minimum points $t_k\to+\infty$ such that
$$
m_k:=W(t_k)\to0,\qquad W'(t_k)=0,
$$
and, after passing to a subsequence,
$$
\frac1{m_k}(\bar\omega_1,\bar\omega_2)(t_k+s)\to
\left(\frac12e^s,\frac12e^{-s}\right)\;\text{or}\;\left(\frac12e^{-s},\frac12e^s\right)\quad\text{in $C^2_{\rm loc}(\mathbb R)$. }
$$
In both cases,
\begin{equation}\label{eq4-20}
\frac{\bar\omega_2(t_k)}{\bar\omega_1(t_k)}\to1.
\end{equation}

Since $N=4$, the averaged equations are
\begin{equation}\label{eq3-36-1}
\begin{cases}
\bar\omega_1''-\bar\omega_1+
(1+\varepsilon_1(t))\left(\mu_1\bar\omega_1^3
+\beta\bar\omega_1\bar\omega_2^2\right)=0,\\
\bar\omega_2''-\bar\omega_2+
(1+\varepsilon_2(t))\left(\mu_2\bar\omega_2^3
+\beta\bar\omega_2\bar\omega_1^2\right)=0,
\end{cases}
\end{equation}
where
$\varepsilon_i(t)=O(e^{-t})$.
Now we split the proof into three cases.

\medskip
\noindent{\bf Case 1.} Assume that
$\mu_1\ge\mu_2$ and $\beta>\mu_2$.

In this case, we
set
$$
r(t):=\frac{\bar\omega_2(t)}{\bar\omega_1(t)},
\qquad L(t):=\log r(t).
$$
A direct computation gives
\be\label{eq4-16}
L''+\left(\frac{\bar\omega_1'}{\bar\omega_1}+
\frac{\bar\omega_2'}{\bar\omega_2}\right)L'
=\frac{\bar\omega_2''}{\bar\omega_2}-
\frac{\bar\omega_1''}{\bar\omega_1},
\ee
with
$$
\begin{aligned}
\frac{\bar\omega_2''}{\bar\omega_2}-
\frac{\bar\omega_1''}{\bar\omega_1}
=&
\bar \omega_1^2\Big[\bigl((1+\varepsilon_1)\mu_1
-(1+\varepsilon_2)\beta\bigr) \\
&+
\bigl((1+\varepsilon_1)\beta
-(1+\varepsilon_2)\mu_2\bigr)r^2\Big] .
\end{aligned}
$$
Then by $\beta>\mu_2$ and $\varepsilon_i(t)\to 0$ as $t\to+\infty$, there exist $q_+>1$ and $T>T_0$ large such that
\be\label{eq4-17}
t\geq T\;\text{and}\;r(t)\ge q_+\quad\Longrightarrow\quad
\frac{\bar\omega_2''}{\bar\omega_2}(t)
-\frac{\bar\omega_1''}{\bar\omega_1}(t)>0.
\ee

We first record the corresponding invariance property. We claim that,  if for some $T_1\ge T$, $r(T_1)>q_+$, $L'(T_1)>0,$ then\be\label{eq4-18}
r(t)>q_+,\qquad L'(t)>0\qquad\text{for all }t\ge T_1.
\ee
Indeed, if $L'$ reaches the first zero at some $s_0>T_1$, then
$L'(s)=\frac{r'(s)}{r(s)}>0$ on $[T_1,s_0)$ and hence $r(s_0)>q_+$,
$L'(s_0)=0$ and $L''(s_0)\le0$. But \eqref{eq4-16}-\eqref{eq4-17} gives
$$
L''(s_0)=\frac{\bar\omega_2''}{\bar\omega_2}(s_0)
-\frac{\bar\omega_1''}{\bar\omega_1}(s_0)>0,
$$
a contradiction. So this forward invariance property holds.

Similarly, we also have a backward invariance property: If for some $T_1> T$, $r(T_1)>q_+,\ L'(T_1)<0,$
then
\be\label{eq4-19}
r(t)>q_+,\qquad L'(t)<0\qquad\text{for all }T\le t\le T_1.
\ee
This is the same argument applied backward in time. More precisely, if $L'$ reaches zeros when one goes backward from $T_1$ to $T$, then at the first time $T\leq s_0<T_1$ we have $r(s_0)>q_+$, $L'(s_0)=0$ and $L''(s_0)\le0$, while \eqref{eq4-16}-\eqref{eq4-17} gives $L''(s_0)>0$, a contradiction again.

We now take the valley sequence $t_k$ in Lemma~\ref{lemma3-2} sufficiently sparse so
that, for the fixed number $S$ chosen below, there are still earlier and later valleys
outside $(t_k-S,t_k+S)$. This does not change the valley profile, after passing
to a subsequence.

Choose $S>0$ so large that $e^{2S}>q_+.$
Suppose first that the valley profile at $t_k$ is
$$
\frac1{m_k}(\bar\omega_1,\bar\omega_2)(t_k+s)
\longrightarrow\left(\frac12e^s,\frac12e^{-s}\right).
$$
Then
$$
r(t_k-S)\to e^{2S}>q_+,\quad L'(t_k-S)\to -2,\quad\text{as }t_k\to+\infty.
$$
Thus, for $k$ sufficiently large,
$r(t_k-S)>q_+,\ L'(t_k-S)<0.$
Then by the backward invariance property just proved, we have
$$
r(t)>q_+>1\qquad\text{for all }T\le t\le t_k-S.
$$
This contradicts with the existence of earlier valleys $t_j<t_k-S$, because \eqref{eq4-20} gives
$r(t_j)=\frac{\bar\omega_2(t_j)}{\bar\omega_1(t_j)}\to1$
along such valleys.

Suppose next that the valley profile at $t_k$ is
$$
\frac1{m_k}(\bar\omega_1,\bar\omega_2)(t_k+s)
\longrightarrow\left(\frac12e^{-s},\frac12e^s\right).
$$
Then
$$
r(t_k+S)\to e^{2S}>q_+,\quad L'(t_k+S)\to2,\quad\text{as}\;t_k\to+\infty,
$$
and so for $k$ sufficiently large,
$r(t_k+S)>q_+,\ L'(t_k+S)>0.$
Then the forward invariance property gives
$$
r(t)>q_+>1\qquad\text{for all }t\ge t_k+S.
$$
Again this contradicts with the existence of later valleys $t_j>t_k+S$ with
$r(t_j)=\frac{\bar\omega_2(t_j)}{\bar\omega_1(t_j)}\to1.$

Therefore, we have proved $W(t)\to 0$ as $t\to+\iy$ for Case 1.

\medskip
\noindent{\bf Case 2.} Assume that
$\mu_2\ge\mu_1$ and $\beta>\mu_1$.

This case is symmetric to Case 1. We only need to revise the definition of $r(t)$ to be $r(t):=\frac{\bar\omega_1(t)}{\bar\omega_2(t)}$, then the rest arguments are the same.

%

\medskip
\noindent{\bf Case 3.} Assume that $\mu_1=\mu_2=\beta.$

In this case, Theorem \ref{thm3-1} is already covered by \cite[Theorem 1.5]{GKS}.
Here, we provide an alternative proof.
We use the exact angular momentum of the original  cylindrical functions:
$$
\mathcal J(t):=\int_{\mathbb S}
\left(\omega_1(\omega_2)_t-\omega_2(\omega_1)_t\right)\,d\theta.
$$
A direct computation from \eqref{eq3-10} gives
{\allowdisplaybreaks
\begin{align*}
\mathcal J'(t)=&\int_{\mathbb S}
\left(\omega_1(\omega_2)_{tt}-\omega_2(\omega_1)_{tt}\right)\,d\theta  \\
=&\int_{\mathbb S}
\left[\omega_1\left(-\Delta_\theta\omega_2+\omega_2-\beta\omega_2(\omega_1^2+\omega_2^2)\right)\right.\\
&\qquad\left.-\omega_2
\left(-\Delta_\theta\omega_1+\omega_1-\beta\omega_1(\omega_1^2+\omega_2^2)\right)\right]\,d\theta\\
=&\int_{\mathbb S}\left(-\omega_1\Delta_\theta\omega_2
+\omega_2\Delta_\theta\omega_1\right)\,d\theta,
\end{align*}
}%
and the integration by parts shows that
$$
\int_{\mathbb S}\omega_1\Delta_\theta\omega_2\,d\theta
=-\int_{\mathbb S}\nabla_\theta\omega_1\cdot\nabla_\theta\omega_2\,d\theta
=\int_{\mathbb S}\omega_2\Delta_\theta\omega_1\,d\theta .
$$
Thus $\mathcal J'=0$, namely $\mathcal J$ is a constant.

On the other hand, by the asymptotic estimates \eqref{eq3-3-1}-\eqref{eq3-3-3},
$$
\mathcal J(t_k)=|\mathbb S|\left(\bar\omega_1\bar\omega_2'
-\bar\omega_2\bar\omega_1'\right)(t_k)+o(m_k^2).
$$
If the valley profile is
$$
\frac1{m_k}(\bar\omega_1,\bar\omega_2)(t_k+s)
\longrightarrow\left(\frac12e^s,\frac12e^{-s}\right),
$$
then
$$
\mathcal J(t_k)=-\frac{|\mathbb S|}{2}m_k^2+o(m_k^2).
$$
If the valley profile is
$$
\frac1{m_k}(\bar\omega_1,\bar\omega_2)(t_k+s)
\longrightarrow\left(\frac12e^{-s},\frac12e^s\right).
$$
then
$$
\mathcal J(t_k)=\frac{|\mathbb S|}{2}m_k^2+o(m_k^2).
$$
Since $m_k\to0$, we have $\mathcal J(t_k)\to0$, so $\mathcal J(t)\equiv 0$. But $\mathcal J(t_k)=\pm\frac{|\mathbb S|}{2}m_k^2+o(m_k^2)\ne0$ for $k$ large, a contradiction again.

In conclusion, 
we have proved  $W(t)\to 0$ as $t\to+\iy$ for all three cases. Applying Lemma~\ref{lemma3-1} again, we conclude that the singularity $0$ is removable for both $u$ and $v$. This completes the proof of Theorem \ref{thm3-1} for $N=4$.
\end{proof}

%

\begin{remark}
Under the stronger assumptions
$$
N\ge5\qquad\text{or}\qquad N=4,\quad \beta>\max\{\mu_1,\mu_2\},
$$
Lemma~\ref{lemma3-3} also has a different proof based on the ratio argument. We explain this alternative proof because it is useful for understanding the mechanism behind the exclusion of a return from a valley to a balanced exit.

First we discuss the right exit $\tau_k^+$. Using the same notations as Case 1 in the proof of Theorem \ref{thm3-1} for $N=4$, a direct computation gives
$$
\begin{aligned}
\frac{\bar\omega_2''}{\bar\omega_2}-
\frac{\bar\omega_1''}{\bar\omega_1}
=&(1+O(e^{-t}))\mu_1\bar\omega_1^{2^*-2}
-(1+O(e^{-t}))\mu_2\bar\omega_2^{2^*-2}\\
&+\beta
\left((1+O(e^{-t}))\bar\omega_1^{\frac{2^*}{2}-2}\bar\omega_2^\frac{2^*}{2}-(1+O(e^{-t}))\bar\omega_2^{\frac{2^*}{2}-2}\bar\omega_1^\frac{2^*}{2}\right),\\
=&\bar\omega_1^{2^*-2} H(t,r(t)),
\end{aligned}
$$
where
\begin{align*}
H(t,r):=&(1+O(e^{-t}))\mu_1-(1+O(e^{-t}))\mu_2 r^{2^*-2}\\
&+\beta(1+O(e^{-t}))r^\frac{2^*}{2}-\beta(1+O(e^{-t}))r^{\frac{2^*}{2}-2}.
\end{align*}
Since $0<{2^\ast-2}<\frac{2^\ast}{2}<2$ for $N\geq 5$ and $\beta>\max\{\mu_1,\mu_2\}$ for $N=4$,
 there exist large constants $q_+>1$ and $T>T_0$ such that
$$
t\geq T\quad\text{and}\quad
r(t)\ge q_+\quad\Longrightarrow H(t,r(t))>0.
$$
Consequently,  the following forward invariance property holds: 
If for some $T_1\ge T$, $r(T_1)>q_+$, $L'(T_1)>0,$ then$$
r(t)>q_+,\qquad L'(t)>0\qquad\text{for all }t\ge T_1.
$$

Now suppose that the conclusion of Lemma~\ref{lemma3-3} fails, namely  $\bar\omega_1(\tau_k^+),\bar\omega_2(\tau_k^+)\ge a_0$
up to a subsequence for some $a_0\in (0, \varepsilon/2)$. 
Since $W(\tau_k^+)=\varepsilon$, we have
$$
\frac{a_0}{\varepsilon}\le r(\tau_k^+)\le\frac{\varepsilon}{a_0}.
$$
We can take the constant $q_+$ larger so that
$q_+>\frac{\varepsilon}{a_0}$.

If the valley profile at $t_k$ is
$$
\frac1{m_k}(\bar\omega_1,\bar\omega_2)(t_k+s)\longrightarrow
\left(\frac12e^{-\delta s},\frac12e^{\delta s}\right),
$$
then we choose $S>0$ so large that $e^{2\delta S}>q_+>\frac{\varepsilon}{a_0}$. 
Clearly
$
r(t_k+S)\to e^{2\delta S}$, $L'(t_k+S)\to 2\delta,
$ so
for large $k$,
$$
r(t_k+S)>q_+,\qquad L'(t_k+S)>0.
$$
On the other hand, we note that $t_k+S<\tau_k^+$ for all large $k$. Indeed, if $\tau_k^+-t_k\le S$ along a subsequence, then the comparison estimate \eqref{eq4-60} and $m_k\to0$ give
$$
\varepsilon=W(\tau_k^+)\le m_k\cosh(\delta (\tau_k^+-t_k))\leq m_k\cosh(\delta S)\to0,
$$
a contradiction. Hence $t_k+S<\tau_k^+$, and
then by the forward invariance property, we get $r(\tau_k^+)>q_+>\frac{\varepsilon}{a_0}$, a contradiction.

If the valley profile at $t_k$ is
$$
\frac1{m_k}(\bar\omega_1,\bar\omega_2)(t_k+s)
\longrightarrow\left(\frac12e^{\delta s},\frac12e^{-\delta s}\right),
$$
we only need to revise the definition of $r(t)$ to be $r(t):=\frac{\bar\omega_1(t)}{\bar\omega_2(t)}$, then the rest arguments are the same,
which also gives a contradiction.

Therefore, the conclusion \eqref{eq4-13} holds for $\tau_k^+$. Again the proof of \eqref{eq4-14} for the left exit $\tau_k^-$ is identical after the change of variable $t-t_k$ by $-(t-t_k)$. This gives an alternative proof of Lemma~\ref{lemma3-3}.
\end{remark}

\subsection*{Acknowledgement} The research of Z. Chen is supported by National Key \text{R\&D} Program of China (Grant 2023YFA1010002) and NSFC (No. 12222109). The research of H. Yang is supported by NSFC (No. 12301140).

\subsection*{Data availability} No data was used for the research described in the article.

\subsection*{Conflict of interest} There is no conflict of interest.


\begin{thebibliography}{100}

\bibitem{Ado} J. Andrade, J. M. do \'{O}, 
Asymptotics for singular solutions to conformally invariant fourth order systems in the punctured ball.
{\it J. Differential Equations} {\bf 413} (2024), 190–239.

\bibitem{ADGW} W. Ao, A. DelaTorre, M. Gonz\'{a}lez, J. Wei, A gluing approach for the fractional Yamabe problem with isolated singularities, {\it J. Reine Angew. Math.} {\bf 763} (2020), 25–78.  

\bibitem{ACDFGW}W. Ao, H. Chan, A. DelaTorre, M. Fontelos, M. Gonz\'{a}lez, J. Wei, On higher dimensional singularities for the fractional Yamabe problem: a non-local Mazzeo-Pacard program, {\it Duke Math. J.,} {\bf 168} (2019), 3297–3411.

\bibitem{Aviles} P. Aviles, On isolated singularities in some nonlinear partial differential equations, {\it Indiana Univ. Math. J.,} {\bf 32}
(1983), 773–791.

\bibitem{CGS} L. Caffarelli, B. Gidas, J. Spruck, Asymptotic symmetry and local behavior of semilinear elliptic
equations with critical Sobolev growth, {\it Comm. Pure Appl. Math.,} {\bf 42} (1989), 271-297.

\bibitem{CJSX} L. Caffarelli, T. Jin, Y. Sire, J. Xiong, Local analysis of solutions of fractional semi-linear elliptic
equations with isolated singularities, {\it Arch. Ration. Mech. Anal.,} {\bf 213} (2014), 245–268.

\bibitem{COS}
R. Caju, J. M. do \'{O}, and A. Santos, Qualitative properties of positive singular solutions to nonlinear elliptic systems with critical exponent, {\it Ann. Inst. H. Poincar\'{e} Anal. Non Lin\'{e}aire,} {\bf 36} (2019), 1575–1601. 

\bibitem{CL1} C. Chen, C.-S. Lin, Local behavior of singular positive solutions of semilinear elliptic equations with Sobolev exponent, {\it Duke Math. J.,} {\bf 78} (1995), 315-334.

\bibitem{CL2} C. Chen, C.-S. Lin, Estimates of the conformal scalar curvature equation via the method of moving planes, {\it Comm. Pure Appl. Math.,} {\bf 50} (1997), 971-1017.

\bibitem{CL3} C. Chen, C.-S. Lin, On the asymptotic symmetry of singular solutions of the scalar curvature equations, {\it Math. Ann.,} {\bf 313} (1999), 229-245.

\bibitem{ChenLi} W. Chen, C. Li, Classification of positive solutions for nonlinear differential and integral systems with critical
exponents, {\it Acta Math. Scientia,} {\bf 29} (2009), 949-960.



\bibitem{ChenLin1} Z. Chen, C.-S. Lin, Removable singularity of positive solutions for a critical
elliptic system with isolated singularity, {\it Math. Ann.} {\bf 363} (2015), 501–523. 


\bibitem{CZ} Z. Chen, W. Zou, Positive least energy solutions and phase separation for coupled
  Schr\"{o}dinger equations with critical exponent, {\it Arch. Ration. Mech. Anal.,} {\bf 205} (2012), 515-551.



\bibitem{CZ2} Z. Chen, W. Zou, Positive least energy solutions and phase separation for coupled
  Schr\"{o}dinger equations with critical exponent: higher dimensional case, {\it Calc. Var. PDE.,} {\bf 52} (2015), 423–467.




\bibitem{FK} R. Frank, T. K\"{o}nig, Classification of positive solutions to a nonlinear biharmonic equation with critical exponent. {\it Anal. PDE,} {\bf 12} (2019), 1101–1113.



\bibitem{F} D. J. Frantzeskakis, Dark solitons in atomic Bose-Einstein condesates:
from theory to experiments, {\it J. Phys. A,} {\bf 43} (2010), 213001.

\bibitem{GKS} M. Ghergu, S. Kim, H. Shahgholian, Isolated singularities for semilinear elliptic systems with power-law nonlinearity, {\it Analysis and PDE,} {\bf 13} (2020), 701-739.

\bibitem{GH} J. Guckenheimer, P. Holmes, Nonlinear oscillations, dynamical systems, and bifurcations of vector fields, Springer-Verlag, New York, 1983. 

\bibitem{GL} Y. Guo, J. Liu, Liouville type theorems for positive solutions of elliptic system in $\RN$, {\it Comm. Partial Differ. Equ.,} {\bf 33} (2008), 263-284.

\bibitem{GHWW} Z. Guo, X. Huang, L. Wang, J. Wei, 
On Delaunay solutions of a biharmonic elliptic equation with critical exponent,
{\it J. Anal. Math.,} {\bf 140} (2020), 371–394.

\bibitem{HanQ} Q. Han, X. Li, Y. Li, Asymptotic expansions of solutions of the Yamabe equation and the
$\sigma_k$-Yamabe equation near isolated singular points. {\it Comm. Pure Appl. Math.,} {\bf 74} (2021), 1915–1970.

\bibitem{HXZ} Z. Han, J. Xiong, L. Zhang, Asymptotic behavior of solutions to the Yamabe equation
with an asymptotically flat metric. {\it J. Funct. Anal.,} {\bf 285} (2023), Paper No. 109982.

\bibitem{KL} Yu. S. Kivshar, B. Luther-Davies, Dark optical solitons: physics and applications, {\it Physics Reports} {\bf 298} (1998), 81-197.

\bibitem{KMPS} N. Korevaar, R. Mazzeo, F. Pacard, R. Schoen, Refined asymptotics for constant scalar curvature metrics with isolated singularites, {\it Invent. Math.,} {\bf 135} (1999), 233-272.

\bibitem{JX} T. Jin, J. Xiong,
Asymptotic symmetry and local behavior of solutions of higher order conformally invariant equations with isolated singularities.
{\it Ann. Inst. H. Poincar\'{e} C Anal. Non Lin\'{e}aire,} {\bf 38} (2021), 1167–1216.

\bibitem{Jost} J. Jost, Partial differential equations. Third edition.
{\it Grad. Texts in Math.,}
Springer, New York, 2013. xiv+410 pp.


\bibitem{Li1996} C. Li,  Local asymptotic symmetry of singular solutions to nonlinear elliptic equations.
{\it Invent. Math.} {\bf 123} (1996), 221–231.




\bibitem{LW1} T. Lin, J. Wei, Ground state of $N$ coupled nonlinear Schr\"{o}dinger equations
in $\R^n$, $n\le 3$, {\it Comm. Math. Phys.,} {\bf 255} (2005), 629-653.

\bibitem{Lions} P.-L. Lions, Isolated singularities in semilinear problems, {\it J. Differential Equations,} {\bf 38} (1980), 441–450.

\bibitem{LW} Z. Liu, Z.-Q. Wang; Multiple bound states of nonlinear Schr\"{o}dinger systems,
{\it Commun. Math. Phys.} {\bf 282} (2008), 721-731.

\bibitem{long2020} W. Long and S. Peng. Positive vector solutions for a Schr\"odinger system with external
 source terms. {\it Nonl. Differ. Equ. Appl. NoDEA.} {\bf 27} (2020), Paper No. 5, 36 pp.

\bibitem{MMP} L. Maia, E. Montefusco, B. Pellacci, Positive solutions for a weakly coupled nonlinear
Schr\"{o}dinger systems, {\it J. Differ. Equ.,} {\bf 229} (2006), 743-767.

 \bibitem{NTTV}
B. Noris, H. Tavares, S. Terracini and G. Verzini;  Uniform H\"{o}lder bounds for nonlinear Schr\"{o}dinger system with strong competition. {\it Comm. Pure Appl. Math.} {\bf 63} (2010), 267-302.

\bibitem{P} L. Perko; 
Differential equations and dynamical systems.
Third edition. Texts Appl. Math., 7. 
Springer-Verlag, New York, 2001. xiv+553 pp.

\bibitem{PS} A. Pistoia, N. Soave, 
On Coron's problem for weakly coupled elliptic systems.
{\it Proc. Lond. Math. Soc.} {\bf 116} (2018), 33–67.

 \bibitem{PT} A. Pistoia and H. Tavares; Spiked solutions for Schr\"odinger systems with sobolev critical exponent: the cases of competitive and weakly cooperative interactions. {\it J. Fixed Point
 Theo. Appl.} {\bf 19} (2017), 407–446.

\bibitem{QS} P. Quittner, P. Souplet, Optimal Liouville-type theorems for noncooperative
elliptic Schr\"{o}dinger systems and applications, {\it Comm. Math. Phys.,} {\bf 311} (2012), 1-19.






\bibitem{S} B. Sirakov, Least energy solitary waves for a system of nonlinear Schr\"{o}dinger equations
in $\R^n$, {\it Comm. Math. Phys.,} {\bf 271} (2007), 199-221.


\bibitem{TZ} S. Taliaferro, L. Zhang, Asymptotic symmetries for conformal scalar curvature equations with singularity, {\it Calc. Var. PDE.,} {\bf 26} (2006), 401-428.

 \bibitem{TV} S. Terracini and G. Verzini,  Multipulse phases in $k$-mixtures of Bose-Einstein condensates. {\it Arch. Ration. Mech. Anal.} {\bf 194} (2009), 717-741.




\bibitem{wei2020} J. Wei and Y. Wu; Ground states of nonlinear Schr\"odinger systems with mixed couplings. {\it J. Math. Pures Appl.} {\bf 141} (2020), 50–88.

\bibitem{YZ} H. Yang, W. Zou, Qualitative analysis for an elliptic system in the punctured space.
{\it Z. Angew. Math. Phys.,} {\bf 71} (2020), Paper No. 47, 21 pp.



\end{thebibliography}
\end{document}